\documentclass[dvipsnames, reqno, a4paper, twoside]{amsart}

\usepackage[main=english]{babel}

\usepackage[new]{old-arrows}

\usepackage{amsmath,amssymb,latexsym, amsfonts}
\usepackage{verbatim}
\usepackage{newcent}
\usepackage{mathbbol}

\usepackage[colorlinks=true,citecolor=MidnightBlue,linkcolor=MidnightBlue,urlcolor=MidnightBlue] 
           {hyperref}

\usepackage{mathabx}
\usepackage{bbm}
\usepackage{pifont}
\usepackage{manfnt}
\usepackage[dvipsnames]{xcolor}
\usepackage{tikz-cd}
\usepackage{mathtools, thmtools} 

\usepackage{tikz}

\usepackage[all,ps]{xy}
\usepackage{color}
\usepackage{stmaryrd}
\usepackage{capt-of}
\usepackage{a4wide}
\usepackage[OT2, T1]{fontenc}
\usepackage[textsize=scriptsize]{todonotes}
\usepackage{scalerel,stackengine}  

\makeatletter
\@namedef{subjclassname@2020}{%
  \textup{2020} Mathematics Subject Classification}
\makeatother

\newcommand{\manda}{\scalebox{.8}{\,\mbox{\ding{228}}\,}}

\newcommand\pig[1]{\scalerel*[5pt]{\big#1}{%
  \ensurestackMath{\addstackgap[1.5pt]{\big#1}}}}

\newcommand{\compactlist}[1]{\setlength{\itemsep}{0pt} \setlength{\parskip}{0pt} \setlength{\leftskip}{-1.#1em}}

\newcommand{\Sum}{\textstyle\sum}

\numberwithin{equation}{section}

\DeclareRobustCommand{\SkipTocEntry}[5]{}

\theoremstyle{plain}

\newtheorem{theorem}{Theorem}[section]

\newtheorem{proposition}[theorem]{Proposition}
\newtheorem{prop}[theorem]{Proposition}
\newtheorem{lemma}[theorem]{Lemma}
\newtheorem{lem}[theorem]{Lemma}

\newtheorem{corollary}[theorem]{Corollary}

\theoremstyle{definition}
\newtheorem{definition}[theorem]{Definition}

\newtheorem{example}[theorem]{Example}
\newtheorem{exs}[theorem]{Examples}
\newtheorem{rem}[theorem]{Remark}
\newtheorem*{rem*}{Remark}

\newtheorem{free text}[theorem]{}

\declaretheorem[name=Theorem, style=italics,
  numbered=no]{theorem*}

\newcommand{\uhhu}{{\scriptscriptstyle{U}}}

 \newcommand{\N}{{\mathbb{N}}}

\newcommand{\frkg}{{\mathfrak g}}

\newcommand{\ga}{\alpha}
\newcommand{\gb}{\beta}

\newcommand{\gD}{\Delta}
\newcommand{\gve}{\varepsilon}

\newcommand{\gl}{\lambda}

\newcommand{\gr}{\rho}

\newcommand{\cI}{{\mathcal I}}
\newcommand{\cJ}{{\mathcal J}}

\newcommand{\Hom}{\operatorname{Hom}}

\newcommand{\Ext}{\operatorname{Ext}}

\newcommand{\id}{{\rm id}}

\newcommand{\Ker}{{\rm Ker}\,}

\newcommand{\g}{{\frkg}}                                                        

\newcommand{\due}[3]{{}_{{#2 }} {#1}_{{ #3}}\,}    

\newcommand{\pl}{\partial}

\newcommand{{\Hl}}{{H^{\ell}}}
\newcommand{{\mHop}}{{m_{H^{\rm op}}}}
\newcommand{{\Hop}}{{H^{\rm op}}}
\newcommand{{\mUop}}{{m_{U^{\rm op}}}}
\newcommand{{\mUopp}}{{m_{\scriptscriptstyle{U^{\rm op}}}}}
\newcommand{{\Uop}}{{U^{\rm op}}}
\newcommand{{\mVop}}{{m_{V^{\rm op}}}}
\newcommand{{\Vop}}{{V^{\rm op}}}
\newcommand{{\Ae}}{{A^{\rm e}}}
\newcommand{{\Be}}{{B^{\rm e}}}
\newcommand{{\Ree}}{{R^{\rm e}}}
\newcommand{{\He}}{{H^{\rm e}}}
\newcommand{{\Aop}}{{A^{\rm op}}}
\newcommand{{\Aope}}{({A^{\rm op}})^{\rm e}}
\newcommand{{\Aopl}}{{A^{\rm op}_\pl}}

\newcommand{{\Bop}}{{B^{\rm op}}}
\newcommand{{\Bopp}}{{\scriptscriptstyle{{B^{\rm op}}}}}
\newcommand{{\Bope}}{({B^{\rm op}})^{\rm e}}
\newcommand{{\Bpl}}{{B_\pl}}

\newcommand{{\op}}{{{\rm op}}}
\newcommand{{\coop}}{{{\rm coop}}}
\newcommand{{\sop}}{{*^{\rm op}}}
\newcommand{{\co}}{{{\rm co}}}

\newcommand{\amoda}{A^{\rm e}\mbox{-}\mathbf{Mod}}                  %
\newcommand{\bmodb}{\mathbf{Mod}\mbox{-}\Be}                     %

\newcommand{\umod}{U\mbox{-}\mathbf{Mod}}                     %
\newcommand{\modu}{\mathbf{Mod}\mbox{-}U}         %
\newcommand{\comodu}{\mathbf{Comod}\mbox{-}U}
\newcommand{\comoduop}{\mathbf{Comod}\mbox{-}U^\op}
\newcommand{\ucomod}{U\mbox{-}\mathbf{Comod}}

\newcommand{\modv}{\mathbf{Mod}\mbox{-}V}         %
\newcommand{\comodv}{\mathbf{Comod}\mbox{-}V}
\newcommand{\vcomod}{V\mbox{-}\mathbf{Comod}}
\newcommand{\vopcomod}{V^\op\mbox{-}\mathbf{Comod}}

\newcommand{\wmod}{W\mbox{-}\mathbf{Mod}}                     %

\newcommand{\moddual}{\mathbf{Mod}\mbox{-}U_*}
\newcommand{\dualcomod}{U_*\mbox{-}\mathbf{Comod}}

\newcommand{\yd}{{}^U_U\mathbf{YD}}

\newcommand{\ydu}{{}_U \hspace*{-1pt}\mathbf{YD}^{\hspace*{1pt}U}}

\newcommand{\ydfh}{{}_{F_h}\hspace*{-2pt} \mathbf{YD}^{\hspace*{.5pt}F_h}}
\newcommand{\yduh}{{}_{U_h}\hspace*{-2pt}\mathbf{YD}^{\hspace*{.5pt}U_h}}

\newcommand{\ydw}{{}_W \hspace*{-1pt}\mathbf{YD}^{\hspace*{1pt}W}}

\newcommand{\vyd}{{}^V\hspace*{-1pt}\mathbf{YD}_{\hspace*{1pt}V}}

\newcommand{\altyd}{{}_{W} \hspace*{-1pt}\mathbf{YD}^{\hspace*{1pt}W}}

\newcommand{\dualyd}{{}^{U_*}\hspace*{-1pt}\mathbf{YD}_{\hspace*{1pt}U_*}}

\newcommand{\yddual}{\mathbf{YD}^{{U^*}}_{{U^*}}}

 \newcommand{\lqua}{{\mathbf{\boldsymbol{\ell}QU}}^{A_h}}
 \newcommand{\rqua}{{\mathbf{\boldsymbol{\it r}QU}}^{ A_h}}
 \newcommand{\lqfa}{{\mathbf{\boldsymbol{\ell}QF}}^{ A_h}}
 \newcommand{\rqfa}{{\mathbf{\boldsymbol{\it r}QF}}^{ A_h}}
\newcommand{\lqur}{{\mathbf{\boldsymbol{\ell}QU}}^{ R_h}}

\newcommand{\lqfr}{{\mathbf{\boldsymbol{\ell}QF}}^{ R_h}}

\newcommand{\lact}{\smalltriangleright}
\newcommand{\ract}{\smalltriangleleft}
\newcommand{\blact}{\blacktriangleright}
\newcommand{\bract}{\blacktriangleleft}

\newcommand{{\gog}}{{G \rightrightarrows G_0}}
\newcommand{{\rra}}{\rightrightarrows}

\newcommand{{\lra}}{\ \longrightarrow \ }
\newcommand{{\lla}}{\ \longleftarrow \ }
\newcommand{{\lma}}{\ \longmapsto \ }

\newcommand{{\bull}}{{\scriptscriptstyle{\bullet}}}
\newcommand{{\qqquad}}{{\quad\quad\quad}}
\newcommand{\Aopp}{{\scriptscriptstyle{\Aop}}}

\newsavebox{\foobox}

\newcommand{\mpact}{\mbox{ \raisebox{-1pt}{\ding{227}} }}

\newcommand{\smap}{{\raisebox{0.3pt}{${{\scriptscriptstyle {[+]}}}$}}}
\newcommand{\smam}{{\raisebox{0.3pt}{${{\scriptscriptstyle {[-]}}}$}}}



\begin{document}

\title{Duality for action bialgebroids}

\author{Sophie Chemla}
\author{Fabio Gavarini}
\author{Niels Kowalzig}

\address{S.C.: Sorbonne Universit\'e \& Universit\'e Paris Cit\'e, CNRS, IMJ-PRG, F-75005 Paris, France}
\email{sophie.chemla@imj-prg.fr}

\address{F.G. \& N.K.: Dipartimento di Matematica, Universit\`a di Roma Tor Vergata, Via della Ricerca
Scientifica 1, 00133 Roma, Italy}
\email{gavarini@mat.uniroma2.it, kowalzig@mat.uniroma2.it}

\begin{abstract}
 We study the effect of linear duality on action bialgebroids (also known as smash product or scalar extension bialgebroids)  and, for those bearing a quantisation nature, the effect of Drinfeld functors underlying the quantum duality principle.
 By means of various categorical equivalences, it is shown that any braided commutative Yetter-Drinfeld algebra over any bialgebroid is also a braided commutative Yetter-Drinfeld algebra over the respective dual bialgebroid.
 This implies that the action bialgebroid of the dual exists, which is then proven to be isomorphic, as a bialgebroid, to the dual of the initial action bialgebroid:
in short, (linear) duality commutes with the action bialgebroid construction.  Similarly, for quantum groupoids to which the Drinfeld duality functors apply and the quantum duality principle holds, these Drinfeld duality functors are shown to commute \hbox{with the action bialgebroid construction as well}.
\end{abstract}

\maketitle

\tableofcontents

\section*{Introduction}

   In the theory of (Lie) groupoids, a first important subclass is that of  {\em action groupoids}, that is, those that express a group action: in fact, one can see this as the first non-trivial example that naturally extends the notion of groups to that of groupoids.  The corresponding infinitesimal notion is that of {\em action Lie algebroid}, that is, a Lie algebroid that arises from a Lie algebra acting on some representation space, see, {\em e.g.}, \cite{Mac:LGALAIDG, CasLadPir:CMFLRA}.

   In a purely algebraic language, groups correspond, loosely speaking, to (special subclasses of) Hopf algebras, and groupoids to Hopf algebroids, or, more in general, to bialgebroids; in particular, action groupoids correspond to the subclass of {\em action bialgebroids}, also known as {\em smash product bialgebroids},  or still {\em scalar extension bialgebroids}, see, {\em e.g.}, \cite{BalSzl:FGEONCB, Sto:SEHA, StoSko:EOSEHAOAUEA}.

   A special feature of the notion of Hopf algebra is its self-duality, in that the linear dual of a (finitely generated) Hopf algebra is again a Hopf algebra (up to technicalities).  This nice behaviour extends to the more general notion of bialgebroids (and even Hopf algebroids) that are also self-dual, up to an important caveat: indeed, the very notion of bialgebroid is two-fold as one defines both  {\em left}  and  {\em right} bialgebroids.  Moreover, every bialgebroid (either left or right) has two natural duals, a  {\em left}  and a  {\em right}  one;
hence, both a left and a right linear duality functor are defined and starting with left or right bialgebroids, this construction eventually gives rise to four linear duality functors.  Finally, these four duality functors switch chirality: the dual of a left bialgebroid $(U,A)$ is a right one, and vice versa.

   All this is well known \cite{KadSzl:BAODTEAD} and fully settled under some projectivity and finiteness assumptions for (one of) the underlying  $ A $-module  structure; if this fails to be, one still has some control by considering notions such as topological bialgebroids, or the like, similar to the Hopf algebra case.
      In fact, this is what happens with universal enveloping algebras  $ V\!L $  and jet bialgebroids $ J\!L $, see \cite{KowPos:TCTOHA}, associated to some Lie-Rinehart algebra  $ (L,A)  $:  assuming that  $ L $  is finitely generated projective over the base algebra  $ A $,  both $ V\!L $  and  $ J\!L $ are not finitely generated projective but they can still be seen as dual to each other, {\em i.e.}, non-degenerately paired via linear duality functors as above, up to some technicalities (including that  $ J\!L $  is a bialgebroid only in a suitable topological sense).

   A second special type of duality arises when dealing with bialgebroids  $ W_h $  that are (formal) quantisations of some  $ V\!L $  or some  $ J\!L $,  for some Lie-Rinehart algebra  $ (L,A)  $ as above: one has two functors  $  W_h \to W_h^\vee  $  and  $  W_h \to W_h^{\prime}  $  such that
   \begin{itemize}
     \compactlist{99}
     \item[--]
if  $ W_h $  is a quantisation of  $ J\!L  $,  then  $ W_h^\vee $  is a quantisation of  $ V(L^*) $,
\item[--]
  if  $ W_h $  is a quantisation of  $ V\!L  $,  then  $ W_h^{\prime} $  is a quantisation of  $ J(L^*) $,
\end{itemize}
   where $ L^* = \Hom_A(L,A)$  denotes the  dual Lie-Rinehart algebra. This phenomenon, known as
   {\em quantum duality principle\/}, sprouts from a key idea of Drinfeld for quantum groups, later extended to quantum groupoids in  \cite{CheGav:DFFQG}.  Note that, in this case, both functors  {\em preserve}  chirality, in that if  $ W_h $  is a left (resp.\ right) bialgebroid, then so are  $ W_h^\vee $ and $W_h^{\prime} $.

   In this paper, we investigate what happens with linear duality (in general) and Drinfeld duality functors (in the quantum setup) when working with action bialgebroids.

   First of all, an action bialgebroid has the form  $  R \# U  $, where  $ (U,A) $  is a left bialgebroid and  $ R $  a braided commutative monoid in the category of Yetter-Drinfeld modules over  $ U $. Theorem \ref{wieheisztdashier}  presents a couple of criteria in order to recognise those bialgebroids that are action bialgebroids. Starting from a right bialgebroid $(V,B)$ and for a
   braided commutative monoid $S$ in a respective Yetter-Drinfeld category, 
a similar construction of an action bialgebroid of the form
$ V \# S $ applies. In Lemma \ref{rosamunde} and Theorem \ref{vaexholm},  we prove (see the main text for details and notation):

\begin{theorem*}
Let $(U,A)$ be a left bialgebroid such that $\due U \lact {}$ is finitely generated $A$-projective and let $U_* := \Hom_A(\due U \lact {}, A)$ be its left dual.
Then, if $R$ is a braided commutative monoid in $\ydu$, it is also so in $\dualyd$ and vice versa;
in particular, 
there is an isomorphism 
  \begin{equation*}
  \Hom_R(R \# U, R) \simeq {U_*}
  \# R,
  \end{equation*}
  of right bialgebroids,
which, if $U$ is a right Hopf algebroid (over a left $A$-bialgebroid),
is an isomorphism of left Hopf algebroids (over right $R$-bialgebroids).
    \end{theorem*}

This expresses the fact that (linear) duality commutes with the action bialgebroid construction.   As for Drinfeld duality,
 we study the effect of the two Drinfeld functors $ (-)^\vee $  and  $ (-)' $  to 
 quantum groupoids 
 that are, in addition, action bialgebroids, say,  $ F_h \# R_h $  and  $ R_h \# U_h $. The outcome is formulated in Theorems \ref{thm: vee} \& \ref{thm: primeoprime}, see again the main text for the mentioned technical details and notation:

 \begin{theorem*}
   \
   \begin{enumerate}
     \compactlist{99}
   \item
     Let $(F_h, A_h)$ be left quantum formal series bialgebroid and 
     $ R_h $ an  $ h $-topologically  complete $ k[[h]]$-algebra
that is a braided com\-mutative monoid 
in  $ \ydfh  $, and 
such that  $  R_h / hR_h  $  is commutative. 
     Then,
up to a technical assumption, 
$ R_h $  is a braided commutative YD~algebra over  $ F_h^\vee $ as well,
and 
there is a canonical isomorphism
$$
R_h  \# F_h^\vee \ \simeq \  ( R_h  \# F_h )^{\vee}
$$
of  topological  left bialgebroids over  $ R_h $.
In particular,  $  R_h \# F_h^\vee  $  is a left quantum universal enveloping bialgebroid over  $ R_h  $.
\item
Let $(U_h,A_h)$ be a left quantum universal enveloping bialgebroid
and 
$ R_h $  an  $ h $-topologically  complete $ k[[h]]$-algebra
that is a braided com\-mutative monoid 
in  $ \yduh  $, and 
such that  $  R_h / hR_h  $  is commutative. 
Then,
up to a technical assumption, 
$ R_h $  is a braided commutative YD~algebra over  $ U'_h $ as well,  and there is a canonical isomorphism
$$
R_h  \# U_h'  \ \simeq \   (R_h  \# U_h )'
$$
of left bialgebroids over  $ R_h $.  In par\-tic\-u\-lar,  $  R_h \# U_h^{\prime}  $  is a left quantum formal series bialgebroid over  $ R_h  $.
     \end{enumerate}
   \end{theorem*}

These statements, in turn, express the fact that Drinfeld functors commute with the action bialgebroid construction as well.  It is worth stressing that for proving  $  (R_h \# U_h)' \simeq R_h \# U_h^{\prime}  $  we actually resort to a linear duality trick, relying on the previous results (suitably adapted to the specific situation).

   Some final words about the organisation of the paper.
 Section \ref{two}  presents some categorical equivalences regarding Yetter-Drinfeld modules and linear duality in the realm of bialgebroids.
Section \ref{three} focuses on action bialgebroids, extending the well-known construction from \cite{BrzMil:BBAD} from Hopf algebras to (right) Hopf algebroids (over left bialgebroids), formulating a right handed-version of this construction, presenting a characterisation of such action bialgebroids, and finally also the above mentioned results on the dual of action bialgebroids. Section \ref{four}
is devoted to action bialgebroids in the quantum   
 framework, namely the one set up in  \cite{CheGav:DFFQG}:  in particular, building upon the results therein about Drinfeld functors and the  \textit{Quantum Duality Principle},  we apply Drinfeld functors to action bialgebroids, thus finding that those functors ``commute'' with the action bialgebroid construction.  
 Finally, the appendices list essentially all technical tools needed in the paper.

\addtocontents{toc}{\SkipTocEntry}
 
\subsection*{Notation}
Let $k$ be a commutative ring,  possibly a field or of characteristic zero. As customary, unadorned tensor products or $\Hom$s refer to those over $k$. Sweedler notation subscripts typically refer to left bialgebroids, superscripts to right \label{rightSw} bialgebroids,
while left comodules are indicated by round brackets, right comodules by square brackets. The three appendices give a detailed account on all employed bialgebroid related notation.

\addtocontents{toc}{\SkipTocEntry}

\subsection*{Acknowledgements}
%
Partially supported by the MIUR Excellence Department Project MatMod@TOV
(CUP:E83C23000330006).
 Both F.\ G.\ and N.\ K.\ are members of the {\em Grup\-po Nazionale per le Strutture Algebriche, Geometriche e le loro
   Applicazioni} (GNSAGA-INdAM).
 N.\ K.\ moreover thanks the {\em Institut de Math\'ematiques de Jussieu -- Paris Rive Gauche}, where part of this work has been achieved, for hospitality and support via the {\em CNRS}; also partially supported by the Swedish Research Council under grant no.~2021-06594 while in residence at the Mittag-Leffler Institute in Djursholm, Sweden, during spring 2025.

\section{Categorical (co)module equivalences}
\label{two}

This section contains a couple of categorical equivalences that will be needed to dualise action bialgebroids in the next section. For all technical details and notation employed regarding bialgebroids, their duals, their modules and comodules, Yetter-Drinfeld modules, and Hopf algebroids, see the three appendices \ref{A}--\ref{C}.

\begin{lem}
\label{sofocle2}
Let $(U,A)$ be a left bialgebroid and $U_*$ and $U^*$ its left and right dual as defined in Appendix \ref{duals}.
\begin{enumerate}
\compactlist{99}
\item
\label{uhu1}
If $(U,A)$ is left and right Hopf and if both $U_\ract$ and $\due U \lact {}$ are right and left projective over $A$, then there is a (strict) equivalence
\begin{equation}
  \label{curious1}
\ucomod \simeq \comodu,
\end{equation}
which induces a (strict) equivalence
$$
\yd \simeq \ydu
$$
of monoidal categories.
\item
  There are (strict) equivalences
$$
\mathbf{Mod}\mbox{-}{U^*} \simeq \ucomod, \qquad  \comodu \simeq \mathbf{Mod}\mbox{-}{U_*}
$$
of monoidal categories,
which under the assumptions of part (i) induce an equivalence
\begin{equation}
  \label{spargelsuppe}
\mathbf{Mod}\mbox{-}{U^*} \simeq \mathbf{Mod}\mbox{-}{U_*}
\end{equation}
of monoidal categories.
\item
\label{uhu2}
If $U_\ract$ resp.\ $\due U \lact {}$ is finitely generated $A$-projective, then
there is a (strict) equivalence
$$
\mathbf{Comod}\mbox{-}U^* \simeq {U\mbox{-}\mathbf{Mod}} \quad \mbox{resp.} \quad {U\mbox{-}\mathbf{Mod}} \simeq U_*\mbox{-}\mathbf{Comod},
$$
of monoidal categories,
and hence, if both projectivity assumptions on $U$ are met, then one has a (strict) equivalence
\begin{equation}
  \label{curious2}
 \mathbf{Comod}\mbox{-}U^* \simeq U_*\mbox{-}\mathbf{Comod}
 \end{equation}
 of monoidal categories.
\item
\label{uhu3}
If $U_\ract$ is finitely generated $A$-projective, then by means of parts (i)~\&~(ii), there is a (strict) braided
 monoidal
equivalence
$$
\yddual \simeq \yd.
$$
with respect to the right dual. If $\due U \lact {}$ instead is finitely generated $A$-projective, then
there is a braided
(strict)
 monoidal
equivalence
$$
\dualyd \simeq \ydu
$$
with respect to the left dual.
\item
Therefore, in case $U$ is both left and right Hopf and also finitely generated $A$-projective in both aforementioned senses, one has a (strict) equivalence
  $$
\yddual \simeq \dualyd
  $$
of braided monoidal categories.
\end{enumerate}
\end{lem}

\begin{proof}
 We limit ourselves to investigate what is happening on the objects of the respective categories and leave all aspects regarding the respective morphisms to the reader. Observe that in all parts the functor that induces the claimed equivalence is the identity on objects.
  \begin{enumerate}
    \compactlist{99}
  \item
    The equivalence of the comodule categories has been given in \cite[Thm.~4.1.1]{CheGavKow:DFOLHA}: let us only recall from {\em loc.~cit.}\ that if $U$ is left Hopf and $U_\ract$ right projective over $A$, as well as $m \mapsto m_{[0]} \otimes_A m_{[1]}$ a right $U$-coaction on a right $A$-module $M$, then
    \begin{equation}
      \label{jamais1}
      \gl_M \colon M \to U \otimes_A M, \quad m \mapsto
      m_{[1]-} \otimes_A m_{[0]} \gve(m_{[1]+}),
    \end{equation}
    defines a left $U$-coaction on $M$;
while if $U$ is right Hopf and $\due U \lact {}$ projective over $A$,  as well as starting from a left coaction $m \mapsto m_{(-1)} \otimes_A  m_{(0)}$ on a left $A$-module $M$, then
        \begin{equation*}
      \rho_M \colon M \to M \otimes_A U, \quad m \mapsto
\gve(m_{(-1)\smap}) m_{(0)} \otimes_A  m_{(-1)\smam},
      \end{equation*}
        defines a right $U$-coaction on $M$.

        Let us now assume that $M \in \ydu$ is a left-right YD module in the sense of \S\ref{ydydydydr}. If $(U,A)$ is both left and right Hopf, then $M$ is a left-left YD module as well with respect to the left $U$-coaction \eqref{jamais1}, that is, fulfils Eq.~\eqref{yd}: indeed, one has,
  \begin{eqnarray*}
&&
    (u_{(1)} m)_{(-1)} u_{(2)} \otimes_A (u_{(1)} m)_{(0)}
\\
  &  {\overset{\scriptscriptstyle{
\eqref{jamais1}
}}{=}}
    &
(u_{(1)} m)_{[1]-} u_{(2)} \otimes_A (u_{(1)} m)_{[0]} \gve \big((u_{(1)} m)_{[1]+} \big)
\\
&
{\overset{\scriptscriptstyle{
}}{=}}
    &
(\gve(u_{(1)}) \lact u_{(2)} m)_{[1]-} u_{(3)}
\otimes_A
(\gve(u_{(1)}) \lact u_{(2)} m)_{[0]} \gve \big((\gve(u_{(1)}) \lact u_{(2)} m)_{[1]+} \big)
\\
&
{\overset{\scriptscriptstyle{
\eqref{ha3}
}}{=}}
&
\big(\gve(u_{(1)}) \blact (u_{(2)} m)_{[1]}\big)_- u_{(3)} \otimes_A (u_{(2)} m)_{[0]} \gve \pig(\big(\gve(u_{(1)}) \blact (u_{(2)} m)_{[1]}\big)_+ \pig)
\\
&
{\overset{\scriptscriptstyle{
\eqref{Tch7}
}}{=}}
&
\big((u_{(2)} m)_{[1]} u_{(1)\smap} u_{(1)\smam} \big)_- u_{(3)} \otimes_A (u_{(2)} m)_{[0]} \gve \pig(\big((u_{(2)} m)_{[1]} u_{(1)\smap} u_{(1)\smam} \big)_+ \pig)
\\
&
{\overset{\scriptscriptstyle{
\eqref{Tch4}
}}{=}}
&
 \big((u_{\smap(2)} m)_{[1]} u_{\smap(1)} u_{\smam} \big)_- u_{\smap(3)} \otimes_A (u_{\smap(2)} m)_{[0]} \gve \pig( \big((u_{\smap(2)} m)_{[1]} u_{\smap(1)}
u_{\smam}\big)_+ \pig)
\\
&
{\overset{\scriptscriptstyle{
\eqref{yd2}
}}{=}}
&
 \big(u_{\smap(2)} m_{[1]}u_{\smam} \big)_- u_{\smap(3)} \otimes_A \big(u_{\smap(1)} m_{[0]}\big) \gve \pig( \big( u_{\smap(2)} m_{[1]}u_{\smam} \big)_+ \pig)
 \\
&
{\overset{\scriptscriptstyle{
\eqref{Sch6}
}}{=}}
&
u_{\smam-} m_{[1]-} u_{\smap(2)-} u_{\smap(3)} \otimes_A \big(u_{\smap(1)} m_{[0]}\big) \gve \big( u_{\smap(2)+} m_{[1]+}u_{\smam+} \big)
\\
&
{\overset{\scriptscriptstyle{
\eqref{Sch3}
}}{=}}
&
u_{\smam -} m_{[1]-} \otimes_A \big(u_{\smap(1)} m_{[0]}\big) \gve \big( u_{\smap(2)} m_{[1]+}u_{\smam +} \big)
\end{eqnarray*}
  \begin{eqnarray*}
&
{\overset{\scriptscriptstyle{
\eqref{mampf3}
}}{=}}
&
u_{(1)} m_{[1]-} \otimes_A \big(u_{(2)\smap(1)} m_{[0]}\big) \gve \big( u_{(2)\smap(2)} m_{[1]+}u_{(2)\smam} \big)
\\
&
{\overset{\scriptscriptstyle{
\eqref{Tch4}
}}{=}}
&
u_{(1)} m_{[1]-} \otimes_A \big(u_{(2)\smap } m_{[0]}\big) \gve \big( u_{(3)} m_{[1]+}u_{(2)\smam} \big)
\\
&
{\overset{\scriptscriptstyle{
}}{=}}
&
u_{(1)} m_{[1]-} \otimes_A \big(u_{(2)\smap } m_{[0]}\big) \gve \big( u_{(3)} m_{[1]+} \bract \gve(u_{(2)\smam}) \big)
\\
&
{\overset{\scriptscriptstyle{
\eqref{Sch9}
}}{=}}
&
u_{(1)} \big(m_{[1]} \bract \gve(u_{(2)\smam})\big)_- \otimes_A \big(u_{(2)\smap } m_{[0]}\big) \gve \big( u_{(3)} \big(m_{[1]} \bract \gve(u_{(2)\smam})\big)_+ \big)
\\
&
{\overset{\scriptscriptstyle{
\eqref{ha3}
}}{=}}
&
u_{(1)} m_{[1]-} \otimes_A \big(u_{(2)\smap} (\gve(u_{(2)\smam}) m_{[0]})\big) \gve \big( u_{(3)} m_{[1]+}\big)
\\
&
{\overset{\scriptscriptstyle{
\eqref{ydforget2}
}}{=}}
&
u_{(1)} m_{[1]-} \otimes_A \big(u_{(2)\smap} s(\gve(u_{(2)\smam})) m_{[0]}\big) \gve \big( u_{(3)} m_{[1]+}\big)
\\
&
{\overset{\scriptscriptstyle{
\eqref{Tch8}
}}{=}}
&
u_{(1)} m_{[1]-} \otimes_A \big(u_{(2)} m_{[0]}\big) \gve \big( u_{(3)} m_{[1]+}\big)
\\
&
{\overset{\scriptscriptstyle{
\eqref{ydforget2}, \eqref{ha3}
}}{=}}
&
u_{(1)} m_{[1]-} \otimes_A u_{(2)} \big(m_{[0]} \gve (m_{[1]+})\big)
\\
&
{\overset{\scriptscriptstyle{
\eqref{jamais1}
  }}{=}}
&
u_{(1)} m_{(-1)} \otimes_A u_{(2)} m_{(0)},
  \end{eqnarray*}
  where we used counitality and the character-like property of a bialgebroid counit in the second, eleventh, and in the penultimate step, along with the Takeuchi property of the bialgebroid coproduct.
   Hence, $M$ also fulfils the defining Eq.~\eqref{yd} for left-left YD modules. The opposite implication is left to the reader.
     \item
  The first two equivalences of  part  (i) appear, {\em e.g.}, in \cite[Prop.~4.2.1]{CheGavKow:DFOLHA}. For later use, we only recall from {\em op.~cit.} that on a right $U$-comodule $M$, a right module structure over the left dual
$U_*$ is defined by:
  \begin{equation}
    \label{onirique}
m\psi := m_{[0]} \langle \psi, m_{[1]} \rangle.
  \end{equation}
  The third equivalence follows from the previous two by combining them with part (i).
\item
    The first equivalence appears in
    \cite[Lem.~4.6]{Kow:WEIABVA},
    the second equivalence follows analogously. For later use, we only state that, if
$\{e_j\}_{1 \leq j \leq n}  \in
    U$ and $ \{e^j\}_{1 \leq j \leq n}  \in {U_*}$ is a dual basis, then
by means of
    \begin{equation}
      \label{ludovisi}
      \begin{array}{rrcl}
        &        m \mapsto m^{(-1)} \otimes_A m^{(0)} &\!\!:=&\!\!
        \Sum_j e^j \otimes_A e_j m,
\\[1mm]
        \mbox{resp.} \quad &
        um &\!\!:=&\!\! \langle m^{(-1)}, u \rangle m^{(0)},
\end{array}
      \end{equation}
one    assigns to any left $U$-action on $M$ a left $U_*$-coaction on $M$, resp.\ to any left $U_*$-coaction $m \mapsto m^{(-1)} \otimes_A m^{(0)}$ on $M$ a left $U$-action on $M$. As a result of both equivalences, if $M$ is a right $U^*$-comodule with coaction denoted $m \mapsto m^{[0]} \otimes_A m^{[1]}$, then
    \begin{equation}
      \label{vendome}
m^{(-1)} \otimes_A m^{(0)} := \Sum_j e^j \otimes_A  m^{[0]} \langle m^{[1]}, e_j \rangle
    \end{equation}
    yields a left $U_*$-coaction, which establishes the equivalence \eqref{curious2}. Observe that for this no (left or right) Hopf structure is needed but see, however, Remark \ref{louvre}.
\item
  The first equivalence still appears in \cite[Lem.~4.6]{Kow:WEIABVA}.
The second one, in turn, is a mirrored version of the first with an analogous proof, which we nevertheless write down here for future reference. To prove the Yetter-Drinfeld property,
  \begin{eqnarray*}
    &&
    \langle \psi^{(2)} (m \psi^{(1)})^{(-1)}, u \rangle \otimes_A (m  \psi^{(1)})^{(0)}
\\
&
{\overset{\scriptscriptstyle{
\eqref{LDMon}
}}{=}}
&
\big\langle (m \psi^{(1)})^{(-1)}, u_{(1)} \ract  \langle \psi^{(2)} , u_{(2)} \rangle\big\rangle \otimes_A (m \psi^{(1)})^{(0)}
\\
&
{\overset{\scriptscriptstyle{
\eqref{onirique}
}}{=}}
&
\big\langle (m_{[0]} \langle \psi^{(1)}, m_{[1]} \rangle)^{(-1)},
u_{(1)}
\ract \langle \psi^{(2)} , u_{(2)} \rangle
\big\rangle \otimes_A (m_{[0]} \langle \psi^{(1)}, m_{[1]} \rangle)^{(0)}
\\
&
{\overset{\scriptscriptstyle{
\eqref{ha2}
}}{=}}
&
\pig\langle {m_{[0]}}^{(-1)} \ract \langle \psi^{(1)}, m_{[1]} \rangle,  u_{(1)}
\ract \langle \psi^{(2)} ,  u_{(2)} \rangle
\pig\rangle
\otimes_A {m_{[0]}}^{(0)}
\\
&
{\overset{\scriptscriptstyle{
\eqref{duedelue}
}}{=}}
&
\pig\langle {m_{[0]}}^{(-1)},  u_{(1)}
\ract \big\langle \psi^{(2)} ,  u_{(2)}
\bract \langle \psi^{(1)}, m_{[1]} \rangle
\big\rangle
\pig\rangle
\otimes_A {m_{[0]}}^{(0)}
\\
&
{\overset{\scriptscriptstyle{
\eqref{trattovideo}
}}{=}}
&
\big\langle {m_{[0]}}^{(-1)},  u_{(1)} \ract
\langle \psi, u_{(2)} m_{[1]} \rangle
\big\rangle \otimes_A {m_{[0]}}^{(0)}
\\
&
{\overset{\scriptscriptstyle{
\eqref{ha2}, \eqref{ludovisi}
}}{=}}
&
\textstyle\sum_j
\langle e^j, u_{(1)} \rangle
\otimes_A
(e_j m_{[0]}) \langle \psi, u_{(2)} m_{[1]} \rangle
\\
&
{\overset{\scriptscriptstyle{
\eqref{schizzaestrappa1}
}}{=}}
&
1 \otimes_A
 (u_{(1)}m_{[0]})\langle \psi, u_{(2)} m_{[1]} \rangle
\\
&
{\overset{\scriptscriptstyle{
\eqref{yd2}
}}{=}}
&
1 \otimes_A
 (u_{(2)}m)_{[0]}  \langle \psi, (u_{(2)} m)_{[1]} u_{(1)} \rangle
\\
&
{\overset{\scriptscriptstyle{
\eqref{trattovideo}
}}{=}}
&
1 \otimes_A
(u_{(2)}m)_{[0]}
\big\langle \psi^{(2)},  (u_{(2)} m)_{[1]}
\bract \langle \psi^{(1)} , u_{(1)} \rangle
\big\rangle
\end{eqnarray*}
  \begin{eqnarray*}
&
{\overset{\scriptscriptstyle{
}}{=}}
&
\langle \psi^{(1)} , u_{(1)} \rangle
\otimes_A
(u_{(2)}m)_{[0]} \big\langle \psi^{(2)}, (u_{(2)} m)_{[1]}  \big\rangle
\\
&
{\overset{\scriptscriptstyle{
\eqref{onirique}
}}{=}}
&
\langle \psi^{(1)} , u_{(1)} \rangle
\otimes_A
(u_{(2)} m) \psi^{(2)}
\\
&
{\overset{\scriptscriptstyle{
\eqref{ludovisi}
}}{=}}
&
\langle \psi^{(1)} , u_{(1)} \rangle
\otimes_A
(\langle m^{(-1)}, u_{(2)} \rangle m^{(0)} ) \psi^{(2)}
\\
&
{\overset{\scriptscriptstyle{
}}{=}}
&
\big\langle \langle m^{(-1)}, u_{(2)} \rangle \lact \psi^{(1)}  , u_{(1)} \big\rangle
\otimes_A m^{(0)} \psi^{(2)}
\\
&
{\overset{\scriptscriptstyle{
\eqref{duedelue}, \eqref{LDMon}
}}{=}}
&
\langle m^{(-1)} \psi^{(1)}, u \rangle \otimes_A m^{(0)} \psi^{(2)},
\end{eqnarray*}
which proves Eq.~\eqref{borghese} for right-left YD modules over right bialgebroids, as desired.
\item
The proof of this part now follows directly from the preceding statements.
  \end{enumerate}
  This concludes the proof of the lemma.
  \end{proof}

\begin{rem}
\label{louvre}
  \
  \begin{enumerate}
    \compactlist{99}
  \item
    The equivalence \eqref{spargelsuppe} can also be proven directly without the detour passing through comodule categories by considering the maps $S^*$ and $S_*$ introduced in \cite[\S4.2]{CheGavKow:DFOLHA}.
  \item
One might be tempted to think that the equivalence
$
 \mathbf{Comod}\mbox{-}U^* \simeq U_*\mbox{-}\mathbf{Comod}
 $
 in \eqref{curious2} is a simple consequence of a right bialgebroid version of the comodule equivalence \eqref{curious1} in part (i) followed by the maps $S^*$ resp.\ $S_*$ just mentioned, which then necessarily requires (left and right) Hopf structures, while the explicit formula \eqref{vendome} given in the proof shows that this is not so. Indeed, if $U$ is left and right Hopf, then both $U^*$ and $U_*$ are so as well as indicated in \S\ref{naemlichhier}, which leads to, for example,
 $
 \mathbf{Comod}\mbox{-}U^* \simeq U^*\mbox{-}\mathbf{Comod};
 $
passing then from $U^*\mbox{-}\mathbf{Comod}$ to $U_*\mbox{-}\mathbf{Comod}$ by the map $S^*$ mentioned above, however, {\em eliminates} the Hopf structure again in the explicit computation, and one reobtains, mildly surprisingly, the simple explicit formula \eqref{vendome} that establishes the equivalence \eqref{curious2}.
\end{enumerate}
  \end{rem}

\noindent By the lemma just proven, we can perform the first step towards dualising action bialgebroids.

\begin{lem}
  \label{rosamunde}
If $R$ is a braided commutative monoid in $\ydu$, then it is also so in $\dualyd$ and vice versa, provided that $\due U \lact {}$ is finitely generated $A$-projective.
\end{lem}

\begin{proof}
That an object $R \in \ydu$ is also an object in $\dualyd$ follows from Lemma \ref{sofocle2}, part (iv). We are left with proving that monoids both in $\umod$ resp.\ $\comodu$ under the equivalences from Lemma \ref{sofocle2}, part (ii), resp.\ part (i), turn into monoids in $\dualcomod $ resp.\ $\moddual $, and finally that the property of (braided) commutativity is preserved.

As for the first statement, assume that $(R, \cdot)$ be a monoid in $\umod$, which manifests itself in Eqs.~\eqref{marmor1} \& \eqref{marmor2}.
Under the equivalence of Lemma \ref{sofocle2}, part (iii), by means of Eq.~\eqref{ludovisi},  this defines a left $U_*$-coaction $\gl_R \colon R \to U_* \times_A R$ on $R$. Pairing then the element
  $
\Sum_j\gl_R(r \cdot r') = e^j \otimes_A e_j (r \cdot r')
  $
with any $u \in U$ in the first tensor factor, yields
$$
\Sum_j \langle e^j, u \rangle \lact e_j (r \cdot r') = u(r \cdot r') =
  (u_{(1)} r)\cdot (u_{(2)} r'),
$$
using Eq.~\eqref{marmor1}.
On the other hand, since the codomain $U_* \times_A R$ of the coaction $\gl_R$ is an $A$-ring, it makes sense to write 
$$
\gl_R(r) \gl_R(r') = \Sum_{j,k} (e^j \otimes_A e_j r)(e^k \otimes_A e_k r')
=
 \Sum_{j,k} e^ke^j \otimes_A (e_j r)\cdot(e_k r'),
$$
using the opposite product in $U_*$ as in Eq.~\eqref{zloty3}.
Pairing this with $u \in U$ yields
\begin{equation*}
  \begin{split}
\Sum_{j,k}  \langle e^k e^j, u \rangle \lact  \big((e_j r)\cdot(e_k r')\big)
&=
\Sum_{j,k}
\big\langle e^j, u_{(1)} \ract \langle e^k, u_{(2)} \rangle\big\rangle \lact (e_j r)\cdot(e_k r')
  \\
  &=
  \Sum_{k}
  \big((u_{(1)} \ract \langle e^k, u_{(2)} \rangle) r\big) \cdot \big(e_k r'\big)
  \\
  &=
  \Sum_{k}
  \big(u_{(1)} r\big) \cdot \big((\langle e^k, u_{(2)} \rangle \lact e_k) r'\big)
 \\
  &=
   (u_{(1)} r ) \cdot (u_{(2)} r'),
\end{split}
\end{equation*}
which is the same as above, and where we used the general property
$$
((u \ract a)r ) \cdot (u' r')
=
((ur) \ract a )\cdot (u' r')
=
(ur) \cdot (a \lact (u'  r'))
=
(ur)\cdot ((a \lact u') r' ),
$$
for any $u, u' \in U$, $r,  r' \in R$ that originates
from Eq.~\eqref{marmor2}.
Finally, from $u1_R = \gve(u) \lact 1_R$ one directly obtains $\gl(1_R) = 1_{U_*} \otimes_A 1_R$.
As a conclusion, $R$ is a monoid in $\dualcomod$ if it is so in $\umod$, and vice versa.

Next, assume that $(R, \cdot)$ is a monoid in $\comodu$, that is,
that Eqs.~\eqref{marmor3} hold.
Then, with respect to the right $U_*$-action \eqref{onirique}, one has:
\begin{equation*}
\begin{split}
  (r \cdot r') \psi &= (r \cdot r')_{[0]} \big\langle \psi, (r \cdot r')_{[1]} \big\rangle
\\
  &= (r_{[0]} \cdot r'_{[0]}) \big\langle \psi, r'_{[1]} r_{[1]} \big\rangle
\\
  &= (r_{[0]} \cdot r'_{[0]}) \big\langle \psi^{(2)} , r'_{[1]} \bract \langle \psi^{(1)}, r_{[1]} \rangle \big\rangle
\\
&= (r_{[0]} \langle \psi^{(1)}, r_{[1]} \rangle) (r'_{[0]} \langle \psi^{(2)} , r'_{[1]} \rangle)
\\
&= (r \psi^{(1)}) \cdot (r' \psi^{(2)}),
\end{split}
  \end{equation*}
where we used the right $U$-comodule Takeuchi property \eqref{ha3} in step four. Moreover,
$$
1_R \, \psi = 1_{[0]} \langle \psi, 1_{[1]} \rangle
= 1_R \bract \langle \psi, 1_U \rangle = 1_R \bract \pl \psi.
$$
Hence, if $R$ is a monoid in $\comodu$, then it is also so in $\moddual$ (and vice versa).
Finally, we have to show that braided commutativity \eqref{marmor4} in $\ydu$ implies an analogous property in $\dualyd$.
One understands that
\begin{eqnarray*}
  r \cdot r'
& \stackrel{\scriptscriptstyle\eqref{marmor4}}{=} &
  r'_{[0]} \cdot (r'_{[1]} r)
\\
 & \stackrel{\scriptscriptstyle\eqref{ludovisi}}{=} &
r'_{[0]} \cdot \big(\langle r^{(-1)}, r'_{[1]} \rangle r^{(0)}\big)
\\
& \stackrel{\scriptscriptstyle\eqref{marmor2}}{=} &
\big(r'_{[0]} \langle r^{(-1)}, r'_{[1]} \rangle\big) \cdot r^{(0)}
\\
& \stackrel{\scriptscriptstyle\eqref{onirique}}{=} &
(r' r^{(-1)}) \cdot r^{(0)},
  \end{eqnarray*}
which is the braided commutativity for $\dualyd$, see Eq.~\eqref{kindertotenlieder}.
The reverse implications follow by analogous arguments along the same lines.
\end{proof}

\section{Action Hopf algebroids and their duals}
\label{three}

In this section, we shall discuss one of our main objects of study, that is, the construction of new bialgebroids out of old ones by combining a smash product with a centre construction, or more precisely, using monoids in the right weak centre of the monoidal category given by left (resp.\ right) modules over a left (resp.\ right) bialgebroid. We will also discuss a recognition theorem to understand when a bialgebroid bears such a structure.

\subsection{Action bialgebroids and action Hopf algebroids}
\label{leftsmash}
In this subsection, we will give a straightforward generalisation to bialgebroids and even (right) Hopf algebroids of the original result in  \cite[Thm.~4.1]{BrzMil:BBAD}, where this was performed for Hopf algebras only as a starting point. This is the content of the subsequent theorem.

\begin{theorem}
\label{tollhaus2}
Let $(U,A)$ be a left bialgebroid and $R$ a braided commutative Yetter-Drinfeld algebra in $\ydu$.
Then $R \# U$ is a left bialgebroid over $R$ and a right Hopf algebroid if $(U,A)$ already was a right Hopf algebroid.
\end{theorem}

\begin{proof}
First, let us fix the tensor product underlying $R \# U$ as $R \, \otimes_A \due U \lact {}$. The product structure on this (seen either as a $k$-algebra or as an $\Ae$-ring) is the customary smash product, {\em i.e.},
  \begin{equation}
    \label{orangen1}
(r \otimes_A u)(r' \otimes_A u')
= r \cdot (u_{(1)} r') \otimes_A u_{(2)} u',
  \end{equation}
  which is well defined by means of Eq.~\eqref{ydforget2}, the $A$-bilinearity of the coproduct, but in particular Eq.~\eqref{marmor2} that guarantees the well-definedness over the Sweedler presentation of the coproduct in $U_\ract \otimes_A \due U \lact {}$. Next, set source and target as
\begin{equation}
  \label{celivre}
  s\colon R \to R \# U, \ r \mapsto r \otimes_A 1,
\qquad   t\colon R \to R \# U, \ r \mapsto r_{[0]} \otimes_A r_{[1]}.
\end{equation}
The proof that their images commute in $R \# U$ is analogous (and hence omitted) to that in \cite[Thm.~4.1]{BrzMil:BBAD} using, in particular, the braided commutativity \eqref{marmor4}, but note that the coaction is here, in contrast to {\em loc.~cit.}, over a bialgebroid and not merely over a bialgebra. These source and target maps, together with the multiplication \eqref{orangen1}, turn $R \# U$ into an $\Ae$-ring. Moreover,
they define the codomain of the left coproduct to be (the Takeuchi subspace of) the tensor product
$
(R \# U)_\ract \otimes_R \due {(R \# U)} \lact {}
$
in which we quotient by
\begin{equation}
  \label{djursholm}
\big((r_{[0]} \otimes_A r_{[1]})  (r' \otimes_A u')\big)
\otimes_R (r'' \otimes_A u'') - (r' \otimes_A u') \otimes_R (r \cdot r'' \otimes_A u'').
\end{equation}
We can therefore set the (left)
coproduct as
\begin{equation}
  \label{mandarinen1}
  \gD
  \colon r \otimes_A u \mapsto (r \otimes_A u_{(1)}) \otimes_R (1 \otimes_A u_{(2)}),
\end{equation}
which is well-defined over the Sweedler presentation thanks to the fact that
$$
r \otimes_A u \ract a = \big((a1_R)_{[0]} \otimes_A (a1_R)_{[1]}\big)(r \otimes_A u) =  t(a) (r \otimes_A u)
$$
as well as
$1 \otimes_A a \lact u = s(a) (1 \otimes_A u)$.
The appurtenant (left) counit simply reads
\begin{equation}
  \label{sicherheitshinweis1}
  \gve \colon R \# U  \to R, \quad r \otimes_A u \mapsto r\gve(u).
\end{equation}
It is a straightforward check (which we skip) that the so-given structure maps turn $R \# U$ into a left bialgebroid over $R$.
More interesting, if $(U,A)$ is, in addition, a right Hopf algebroid (over a left bialgebroid) in the sense of  \S\ref{sokc}, then $R \# U$ turns itself into a right Hopf algebroid as well:
the translation map which induce the inverse of the Hopf-Galois map
$$
(r \otimes_A u) \otimes_R (r' \otimes_A u')
\to
\big((r \otimes_A u)_{(1)} (r' \otimes_A u')\big) \otimes_R (r \otimes_A u)_{(2)}
$$
is given by
 \begin{equation}
   \label{erkaeltet}
\ga_r^{-1} \colon r \otimes_A u \mapsto (r_{[0]} \otimes_A u_{\smap }) \otimes_R (1 \otimes_A u_{\smam  } r_{[1]}),
\end{equation}
and even here we omit the explicit check that Eqs.~\eqref{Tch2} \& \eqref{Tch3}, which express the required invertibility, are fulfilled as these follow directly from the respective property of $U$, along with Eqs.~\eqref{orangen1}, \eqref{celivre}, as well as \eqref{djursholm}.
\end{proof}

\begin{definition}
A (left) bialgebroid of the form $(R\#U, R)$ as above, where $(U,A)$ is itself a (left) bialgebroid, will be called a {\em (left) action bialgebroid}, sometimes referred to as {\em smash product bialgebroid} or still {\em scalar extension bialgebroid}.
\end{definition}
  
\begin{example}
\label{sable}
  For a left bialgebroid $(U,A)$ such that $\due U \lact {}$ is left $A$-projective, let
  $$
  \widetilde U :=
\due U \lact \bract \otimes_\Ae A =
  \big\{ u \in U \mid a \lact u = u \bract a, \ \forall\, a \in A \big\},
$$
that is, the subset of those elements in $U$ that commute with the source map.
This subset is, in particular, an $A$-bimodule by means of the target map, which we denote, as always, by $\due {\widetilde U} \blact \ract$.
Observe that $\widetilde U$ is, in general, not a coring anymore, but at least a right $U$-comodule via the map
\begin{equation}
  \label{bonjour}
\rho_{\widetilde{U}} \colon \widetilde U \to \widetilde U_\ract \otimes_A \due U \lact {}, \quad u \mapsto u_{[0]} \otimes_A u_{[1]} :=
 u_{(1)} \otimes_A u_{(2)}
\end{equation}
  induced by the coproduct on $U$. Indeed, if $u \in \widetilde U$, then
  $$
a \lact u_{(1)} \otimes_A u_{(2)} =  \rho_{\widetilde{U}}(a \lact u) = \rho_{\widetilde{U}}(u \bract a) = u_{(1)} \bract a \otimes_A u_{(2)},
$$
and therefore $u_{(1)} \otimes_A u_{(2)} \in \widetilde U \otimes_A U$ by the assumed $A$-projectivity,
while the necessary coassociativity and counitality relations simply follow by those of the coproduct.

Let now $U$ be a right Hopf algebroid (over a left bialgebroid). In contrast to $U$ itself, $\widetilde U$ admits a left adjoint action
\begin{equation}
  \label{bonsoir}
\mpact \colon U \otimes \widetilde U \to \widetilde U, \quad (w, u) \mapsto w_{\smap} u w_{\smam}
\end{equation}
of $U$ on $\widetilde U$. This map is well-defined by mere construction, it is a left action on $\widetilde U$ by using Eqs.~\eqref{Tch1} \& \eqref{Tch6}, and it descends to an $\Ae$-balanced action $\mpact \colon \due U \blact \bract \otimes_\Ae \due {\widetilde U} \blact \ract  \to \widetilde U$ de\-noted by the same symbol,
as seen from \eqref{Tch9}. It is a straightforward check now (which we omit) that $\widetilde U$ is a left-right YD module over $U$. Observe that if $U$ is a Hopf algebra (with invertible antipode), by $\widetilde U = U$, this reduces to one of the customary ways of seeing a Hopf algebra as a YD module over itself.

The braided commutative monoid to consider in $\ydu$ in the sense of \S\ref{trecce1} now rather is $\widetilde U^\op$, equipped with its opposite multiplication as is immediately clear from the requirement \eqref{marmor3} of being a monoid in $\comoduop$. One has, for $w \in U$ and $ u, u' \in \widetilde U$
\begin{eqnarray*}
  (w_{(1)} u) \cdot (w_{(2)} u')
  & \stackrel{\scriptscriptstyle\eqref{bonsoir}}{=} &
(w_{(1)\smap} u w_{(1)\smam}) \cdot (w_{(2)\smap} u' w_{(2)\smam})
\\
&  = &
w_{(2)\smap} u' w_{(2)\smam} w_{(1)\smap} u w_{(1)\smam}
\\
& \stackrel{\scriptscriptstyle\eqref{Tch4}}{=} &
w_{\smap (2)\smap} u' w_{\smap (2)\smam} w_{\smap (1)} u w_{\smam}
\\
& \stackrel{\scriptscriptstyle\eqref{Tch3}}{=} &
w_{\smap} u' u w_{\smam}
\\
& \stackrel{\scriptscriptstyle\eqref{bonsoir}}{=} &
w \mpact (u \cdot u'),
\end{eqnarray*}
which proves, for example, \eqref{marmor1}, and likewise
$$
u_{[0]} \cdot (u_{[1]}u') \stackrel{\scriptscriptstyle\eqref{bonsoir}}{=}  u_{[0]} \cdot (u_{[1]\smap}u' u_{[1]\smam}) \stackrel{\scriptscriptstyle\eqref{bonjour}}{=}  u_{(2)\smap}u' u_{(2)\smam}u_{(1)} \stackrel{\scriptscriptstyle\eqref{Tch3}}{=} uu' = u' \cdot u,
$$
which proves \eqref{marmor4}. As an outcome, $\widetilde U^\op \# U$ is a left bialgebroid, and since $U$ was already a right Hopf algebroid, $\widetilde U^\op \# U$ is so as well by means of the translation map in \eqref{erkaeltet}. 
If one considers a bialgebroid, in a certain sense, as a cogroupoid kind of object, the construction $\widetilde U^\op \# U$ might be termed {\em Weyl Hopf algebroid}, in analogy to the classical Weyl groupoid $G \rtimes G$  (as termed so in \cite[\S6.4]{HanMaj:BCHA}),
which will be further discussed in Example~\ref{pomeriggio}.
In case the bialgebroid $(U,A) = (H,k)$ is actually a Hopf algebra over a field $k$, then the construction of $\widetilde U^\op \# U$ reduces to the one of $H^\op \# H$ presented in {\em op.~cit.}

\end{example}

\begin{example}
  \label{toujours}
 Let $(U,A)$ be a right Hopf algebroid (over a left bialgebroid) such that $\due U \lact {}$ is finitely generated projective over $A$. Then it is known, see \S\ref{naemlichhier}, that the left dual $(U_*, A)$ is a left Hopf algebroid (over a right $A$-bialgebroid), with explicit expression of the translation map  given by 
\begin{equation*}
    \psi^\smam \otimes_A \psi^\smap := \textstyle\sum_j e^j \otimes_A (e_j \manda \psi)
\end{equation*}
for $\psi \in U_*$,
where
$
(u \manda \psi)(u')
:= \gve(u_\smap \bract \phi(u_\smam u'))
$
for $u ,  u' \in U$, and
$\{e_j\}_{1 \leq j \leq n} \in U$, $\{ e^j\}_{1 \leq j \leq n} \in U_*$ is a dual basis. This allows to mimic the construction from the preceding Example \ref{sable} in a right bialgebroid manner for $U_*$, that is, one defines
  $$
  \widetilde U_* :=
{}_\lact {(U_*)} {}_\bract \otimes_\Ae A =
  \big\{ \psi \in U_* \mid a \lact \psi = \psi \bract a, \ \forall\, a \in A \big\},
$$
  on which there is an adjoint {\em right} $U_*$-action
  \begin{equation*}
  \reflectbox{\mpact}
    \colon \widetilde U_* \otimes U_* \to \widetilde U_*, \quad (\tilde \psi, \psi) \mapsto \psi^{\smam} \tilde \psi  \psi^{\smap}.
\end{equation*}
  It is then again (and omitted again) a direct check that this right action along with the left coaction induced from the coproduct $\gD_r$ on $U_*$ as in Eq.~\eqref{weiszjanich}, analogous to the considerations above, endow $\widetilde U_*$ with the structure of a right-left YD module, and, more interesting, even a braided commutative monoid in $\dualyd$.

  Combining the second equivalence in part (iv) of Lemma \ref{sofocle2} with Lemma \ref{rosamunde}, we obtain that $\widetilde U_*$ is a braided commutative monoid in $\ydu$, and therefore $\widetilde U_* \# U$ is a right Hopf algebroid (over a left bialgebroid), which might once more be termed {\em Weyl Hopf algebroid}.
Again,  in case the bialgebroid $(U,A) = (H,k)$ is actually a Hopf algebra over a field $k$, then the construction of $\widetilde U_* \# U$ reduces to the one of $(H^\op)^* \# H$ presented in \cite[\S6.4]{HanMaj:BCHA}.
\end{example}

\subsection{Recognising action bialgebroids}
A natural question is when or if a given bialgebroid can be singled out as being isomorphic to an action bialgebroid. This is the content of the following theorem which permits to distinguish these:

\reversemarginpar

\begin{theorem}
  \label{wieheisztdashier}
 A left bialgebroid
 $(U, B, s, t, \gD, \gve)$
  is an action left bi\-al\-ge\-broid
  if and only if: 
  \begin{enumerate}
\compactlist{50}
  \item
there is a left bialgebroid
  $(W,A)$
  along with a morphism $(F, f) \colon   (W, A) \to (U,B)$ of left bialgebroids;
\item
  there is a left $B$-linear right $W$-coaction $\rho_U \colon U \to U_{\ract f} \,  \otimes_A W $ on $U$
  in the sense of $\rho_U(b \lact u) = b \lact \rho_U(u)$, satisfying, in addition,
    \begin{equation}
      \label{atlas}
(U \otimes_B F) \circ \rho_U = \gD,
    \end{equation}
    where the map $U \otimes_B F$ is the composition $U \otimes_A W \longhookrightarrow
    U \otimes_B (B_f \otimes_A W) \xrightarrow{U \otimes_B \xi} U \otimes_B U$, with $\xi$ as in \eqref{dimanche} right below.
      \item
        the map $F \colon W \to U$ is 
        colinear.
\end{enumerate}
If all of the above holds, 
   one has an isomorphism
 \begin{equation}
 \label{dimanche}
    \xi \colon
     B_f \otimes_A W \to
    U, \quad
    b \otimes_A w \mapsto b \lact F(w)
    \end{equation}
of left $B$-modules, with inverse
$u \mapsto \gve(u_{[0]}) \otimes_A u_{[1]}$. Finally,
$B$ becomes a braided commutative monoid in $\altyd$ so that $\xi$
induces an isomorphism
$ (B \# W, B) \to
(U,B)
$
of left bialgebroids.
   \end{theorem}

\begin{proof}
  For ease of notation, the structure maps $s, t, \gD, \gve$ of $U$ will not have a subscript, while those of $W$ do. Moreover, $U_{\ract f}$ indicates $U$ seen as a right $A$-module by means of $t \circ f = F \circ t_W$ multiplied from the left.
%
It is easy to see that $\xi$ from \eqref{dimanche} is an isomorphism (of left  $B$-modules and actually of left bialgebroids as shown later on):
  to start with, $\xi \circ \xi^{-1} = \id_U$ is an elementary check using \eqref{atlas}, while using the assumed colinearity
  $
   F(w)_{[0]} \otimes_A F(w)_{[1]} =
      F(w_{(1)}) \otimes_A w_{(2)}
  $
      of $F$ along with the left $B$-linearity
      $
      (b \lact u)_{[0]} \otimes_A (b \lact u)_{[1]}
      =
      b \lact u_{[0]} \otimes_A u_{[1]}$ of~$\rho_U$, 
we compute
\begin{equation*}
  \begin{split}
  (\xi^{-1}
  \circ \xi)(b \otimes_A w)
    &
=
\gve\pig(\big(b \lact F(w)\big)_{[0]}\pig) \otimes_A \big(b \lact F(w)\big)_{[1]}
\\
&
= \gve\big(b \lact  F(w)_{[0]}\big) \otimes_A F(w)_{[1]}
\\
&
= b\,\gve\big(F(w_{(1)})\big) \otimes_A w_{(2)}
= bf\big(\gve_W(w_{(1)})\big) \otimes_A w_{(2)} = b \otimes_A w,
  \end{split}
  \end{equation*}
on $B_f \otimes \due W \lact {}$,
using left $B$-linearity of $\gve$ in the third step and $\gve \circ F = f \circ \gve_W$ in the fourth.
Next, the base algebra $B$ of $U$ becomes a left $W$-module by means of the counit of $U$,
  \begin{equation}
    \label{krabbensalat1}
  wb := \gve \big(F(w) \bract b\big) \in B.
  \end{equation}
Combining this with the
 forgetful functor $\wmod \to \amoda$, using $F \circ s_W = s \circ f$ and
$F \circ t_W = t \circ f$, one obtains:
  \begin{equation}
    \label{krabbensalat2}
    a \lact b \ract a' = \gve \big(F(s_W(a)t_W(a')) \bract b\big)
    = f(a)bf(a').
  \end{equation}
  So equipped, $B$ evidently becomes a left $W$-module algebra as follows directly from $F$ being a coring morphism.
  Moreover,
  the injectivity of the source map $s \colon B \to U$ induces an isomorphism
   \begin{equation}
    \label{krabbensalat3}
   B \simeq s(B) \simeq U ^{\hskip 1pt {\rm co} {\hskip 1pt} W} := \{ u \in U \mid \rho_U(u) = u \otimes_A 1_W \},
 \end{equation}
where the second isomorphism holds thanks to the property $F(1_W) = 1_U$ along with Eq.~\eqref{atlas}.

If we now equip $B_f \otimes_A  \due W \lact {}$ with a smash prod\-uct structure as in \eqref{orangen1} with respect to the left $W$-action \eqref{krabbensalat1}, then $\xi$ can be promoted to an isomorphism
\begin{equation*}
\xi \colon B \# W  \to U, \quad b \otimes_A w \mapsto b \lact F(w)
  \end{equation*}
of algebras as easily seen: indeed, on one hand, one has, using \eqref{orangen1},
$$
\xi\big((b \otimes_A w)(b' \otimes_A w') \big) = \xi\big(b \gve(F(w_{(1)}) \bract b') \otimes_A w_{(2)} w' \big) = \big(b \gve(F(w_{(1)}) \bract b')\big) \lact F(w_{(2)} w'),
$$
while
\begin{equation*}
  \begin{split}
\xi(b \otimes_A w)\xi(b' \otimes_A w') = s(b)F(w)s(b')F(w')
&= s(b)s\gve\big(F(w)_{(1)} \bract b' \big) F(w)_{(2)} F(w')
\\
&
= \big(b \gve(F(w_{(1)}) \bract b' )\big) \lact F(w_{(2)} w'),
  \end{split}
  \end{equation*}
using the fact that $F$ is a morphism of rings and corings.
Next, we want to prove that on $B$ (seen as right $A$-module via $f$ as above) a right $W$-coaction can be defined.
To this end, let us consider the element
\begin{equation}
  \label{raumplan}
 b_{[0]} \otimes_A b_{[1]} :=  \xi^{-1}(t(b))
\end{equation}
that is, $t(b) = b_{[0]} \lact F(b_{[1]})$, and show in the following that
$$
\rho_B \colon B \to B_f \otimes_A \due W \lact {}, \quad b \mapsto b_{[0]} \otimes_A b_{[1]}
$$
is a right $W$-coaction on $B$, indeed. Let us first prove that the left $A$-action defined below on $B$, and which will play the r\^ole of the induced action on $B$ as a right comodule, coincides with the action via $f$:
\begin{equation}
  \label{krabbensalat4}
  \begin{split}
ab &:= b_{[0]} f\big(\gve_W(b_{[1]} \bract a)\big)
= b_{[0]} \gve\big(F(b_{[1]} \bract a)\big)
= b_{[0]} \gve\big(F(b_{[1]}) \bract f(a)\big)
\\
&=
 \gve\big(b_{[0]} \lact F(b_{[1]}) \bract f(a) \big)
=
\gve\big(t(b) s(f(a))\big)
\\
&=
f(a) b.
  \end{split}
\end{equation}
With this, it is a straightforward check that $\rho_B$ is $A$-bilinear:
\begin{equation*}
  \begin{split}
\rho_B(aba') &= \rho_B\big(f(a)bf(a')\big)
=  \xi^{-1}\big(t(f(a)bf(a'))\big)
\\
&= \xi^{-1}\big(F (t_W(a'))\big) \, \xi^{-1}(b) \, \xi^{-1}\big(F (t_W(a))\big)
\\
&= (1 \otimes_A t_W(a'))( b_{[0]} \otimes_A b_{[1]}) (1 \otimes_A t_W(a))
\\
&= b_{[0]} \otimes_A t_W(a') b_{[1]}t_W(a)
  \end{split}
  \end{equation*}
since $\xi$ is a ring morphism,
using the multiplication \eqref{orangen1} in $B \# W$, and which expresses the $A$-bilinearity of the right coaction, see Eq.~\eqref{ha3}.
From this it also follows that
$$
b_{[0]} \otimes_A b_{[1]}s_W(a) =  (b_{[0]} \otimes_A b_{[1]}) ( 1 \otimes_A s_W(a)) = f(a) b_{[0]} \otimes_A b_{[1]}
$$
by applying the isomorphism $\xi$ to it and using the fact that source and target in $U$ commute; this expresses the Takeuchi property $\rho_B(b) \in B_f \times_A \due W \lact {}$ from \eqref{ha3} again.
As for coassociativity of $\rho_B$, applying the coproduct in $U$ to both sides of the equation $t(b) = b_{[0]} \lact F(b_{[1]})$ yields, by standard left bialgebroid identities, on its left hand side
$$
1_U \otimes_B t(b) = 1_U \otimes_B b_{[0]} \lact F(b_{[1]}) = t(b_{[0]}) \otimes_B F(b_{[1]}) = b_{[0][0]} \lact F(b_{[0][1]}) \otimes_B F(b_{[1]}),
$$
and on its right hand side
$$
b_{[0]} \lact F(b_{[1]})_{(1)} \otimes_B F(b_{[1]})_{(2)} =  b_{[0]} \lact F(b_{[1](1)}) \otimes_B F(b_{[1](2)}),
$$
by the fact that $F$ is, in particular, a morphism of corings. Therefore, 
$$
b_{[0][0]} \lact F(b_{[0][1]}) \otimes_B F(b_{[1]}) = b_{[0]} \lact F(b_{[1](1)}) \otimes_B F(b_{[1](2)}).
$$
Hence, if $\xi \otimes_B F$ were injective, the coassociativity 
$
b_{[0][0]} \otimes_A b_{[0][1]} \otimes_A b_{[1]} = b_{[0]} \otimes_A b_{[1](1)} \otimes_A b_{[1](2)}
$
followed: indeed, the map $\xi \otimes_B F$ is defined via the commutative diagram
$$
\begin{tikzcd}
  (B \otimes_A W) \otimes_A W \arrow[rr, "\xi \otimes_B F"] \arrow[d, "\simeq"] && U \otimes_B U \arrow[d, equal, double]
  \\
  (B \otimes_A W) \otimes_B (B \otimes_A W) \arrow[rr, "\xi \otimes_B \xi"] && U \otimes_B U,
\end{tikzcd}
$$
where the right $B$-action on $B \otimes_A W$ is given by left multiplication with $\xi^{-1}\big(t(b)\big)$,
one understands that, as above, $\xi \otimes_B F$ essentially amounts to the map $\xi \otimes_B \xi$, which is injective.

As a next step, counitality, that is,
$
b_{[0]} f (\gve_W(b_{[1]})) = b,
$
then follows immediately from
\begin{equation*}
  \begin{split}
t\big(b_{[0]} f (\gve_W(b_{[1]}))\big)
&= t\big(f (\gve_W(b_{[1]}))\big) t(b_{[0]})
\\
&=
 b_{[0][0]} \lact F(b_{[0][1]}) \ract f \big(\gve_W(b_{[1]})\big)
\\
&=
 b_{[0]} \lact  F\big(b_{[1](1)} \ract \gve_W(b_{[1](2)})\big)
\\
&=
b_{[0]} \lact F(b_{[1]})
\\
&= t(b),
  \end{split}
  \end{equation*}
by the fact that source and target commute and then using the injectivity of~$t$. Next, let us show that $B$ is braided commutative in the sense of Eq.~\eqref{marmor4} with respect to the braiding induced by the right $W$-coaction \eqref{raumplan}. Using $s(b')t(b) = t(b)s(b')$ again,
we see that
$$
\xi^{-1}\big(s(b')t(b)\big) = \xi^{-1}\big(s(b')s(b_{[0]})F(b_{[1]})\big)
= \xi^{-1}\big(s(b' b_{[0]})F(b_{[1]}) \big) = b' b_{[0]} \otimes_A b_{[1]},
$$
while
$$
\xi^{-1}\big(t(b)s(b')\big) = \xi^{-1}\big(s(b_{[0]}) F(b_{[1]}) s(b')\big)
= (b_{[0]} \# b_{[1]})(b' \otimes_A 1_W) = b_{[0]} (b_{[1]}b') \otimes_A b_{[2]}.
$$
By applying $B \otimes_A \gve$ on both equations, one concludes, by using the middle of the base algebra action properties in \eqref{marmor2} for module algebras (as mentioned above), the braided commutativity $b' b = b_{[0]} (b_{[1]}b')$. From this, braided commutativity follows in a similar way that $B$ is a $W^\op$-comodule
algebra: since the target map is an anti ring morphism, one has
\begin{equation*}
  \begin{split}
    (bb')_{[0]} \otimes_A (bb')_{[1]}
    &=
    \xi^{-1}\big(t(bb')\big)
= \xi^{-1}(t(b'))\xi^{-1}(t(b))
\\
& =
(b'_{[0]} \otimes_A b'_{[1]})(b_{[0]} \otimes_A b_{[1]})
= b'_{[0]} (b'_{[1]}b_{[0]}) \otimes_A b'_{[2]} b_{[1]}
= b_{[0]} b'_{[0]} \otimes_A b'_{[1]} b_{[1]}.
\end{split}
\end{equation*}
Finally, let us check that $B$ is a left-right YD module over $W$. That the two $A$-bimodule structures on $B$ induced by the left $W$-action resp.\ right $W$-coaction coincide follows directly from Eqs.~\eqref{krabbensalat2} \& \eqref{krabbensalat4}. As for the compatibility \eqref{yd2} between action and coaction, let us apply $\xi$ to the left hand side of Eq.~\eqref{yd2}:
\begin{equation*}
  \begin{split}
\xi\big(w_{(1)} b_{[0]} \otimes_A w_{(2)} b_{[1]}\big)
&= (w_{(1)} b_{[0]}) \lact F\big(w_{(2)} b_{[1]} \big)
\\
&= \gve \big(F(w_{(1)}) \bract b_{[0]} \big) \lact  F(w_{(2)})F(b_{[1]})
\\
&= \gve \pig(\big(F(w) \bract b_{[0]}\big)_{(1)}\pig) \lact \big( F(w) \bract b_{[0]}\big)_{(2)} F(b_{[1]})
\\
&= F(w) s(b_{[0]}) F(b_{[1]})
\\
&= F(w) t(b).
  \end{split}
\end{equation*}
Doing the same with respect to the right hand side in \eqref{yd2}, we obtain:
\begin{equation*}
  \begin{split}
    \xi\big((w_{(2)} b)_{[0]} \otimes_A (w_{(2)} b)_{[1]}  w_{(1)} \big)
    &=  (w_{(2)} b)_{[0]} \lact F\big((w_{(2)} b)_{[1]}  w_{(1)} \big)
    \\
    &    =  (w_{(2)} b)_{[0]} \lact F\big((w_{(2)} b)_{[1]}\big) F( w_{(1)})
    \\
    &    =  t(w_{(2)} b) F( w_{(1)})
\\
    &    =  t\big(\gve(F(w)_{(2)} \bract b)\big) F( w)_{(1)}
\\
    &    =  F( w) t(b)
  \end{split}
\end{equation*}
by the Takeuchi property and counitality in the last step; this is the same as above, as desired.
To sum up, $B$ is a braided commutative monoid in $\ydw$, and therefore $B \# W$ becomes a left bialgebroid according to the construction \S\ref{leftsmash}. It remains to show that $\xi$ is an isomorphism of left bialgebroids. That it is so for the ring structure was shown above, and that it also intertwines all remaining bialgebroid structure maps is a straightforward check, the details of which we omit.

The opposite implication, that is, starting from $U = B \# W$, for $B$ a braided commutative monoid in $\ydw$, is rather obvious: define $\rho_U := B \otimes_A \gD_W$ and $F \colon w \mapsto 1_B \otimes_A w$. Then, using Eq.~\eqref{orangen1}, the relation \eqref{atlas} holds, $F$ is clearly right $W$-colinear, and $U$ is a right $W$-comodule algebra by coassociativity.
\end{proof}

\subsection{Right action bialgebroids}
\label{rightsmash}
In order to deal with duals later on, we need a construction analogous to the above left action bialgebroids in \S\ref{leftsmash} starting from {\em right} bialgebroids as the dual of a left bialgebroid is a right bialgebroid, {\em cf.}\ \S\ref{duals}.
Hence, let $(V,B, s, t, \gD, \pl)$ be a right bialgebroid, and $(R, \cdot)$ a braided commutative monoid in $\vyd$ as detailed in \S\ref{trecce2}.
%
To avoid order flipping operations, let us take $V \otimes_B R := V_\bract \otimes_B R$ as the underlying tensor product for the right action bialgebroid $V \# R$ in this order that we are going to construct by  equipping it with the following operations: its multiplication (as a $k$-algebra) is given by
\begin{equation}
  \label{narrateur}
(v \otimes_B r)(\tilde v \otimes_B r') := v\tilde v^{(1)} \otimes_B (r   \tilde v^{(2)} ) \cdot r',
\end{equation}
where the right $V$-action on $R$ is simply denoted by juxtaposition $R \otimes V \to R, \ (r,v) \mapsto rv$, and where superscript Sweedler notation refers to {\em right} bialgebroid structure maps as mentioned on page \pageref{rightSw}.
We turn
$V \# R$ into a $\Be$-ring by setting
\begin{equation}
  \label{beethovensonata23}
s\colon R \to V \# R, \ r \mapsto 1 \otimes_B r, \qquad t \colon R \to V \# R, \ r \mapsto r^{(-1)} \otimes_B r^{(0)},
\end{equation}
for the source and target map,
which also allow to set up the codomain of the coproduct; which, in turn, is given by
\begin{equation}
  \begin{split}
  \label{mandarinen2}
  V \otimes_B R &\to (V \otimes_B R)_\bract \otimes_R \due {(V \otimes_B R)} \blact {}, \quad
v \otimes_B r \mapsto  (v^{(1)} \otimes_B 1) \otimes_R (v^{(2)} \otimes_B r),
\end{split}
  \end{equation}
whereas for the right counit one takes:
\begin{equation*}
  V \otimes_B R \to R, \quad v \otimes_B r \mapsto \pl(v)r.
  \end{equation*}
It is a direct check (which we omit) analogous to the left case that the resulting object becomes a right bialgebroid over $R$. Again, this right bialgebroid can be promoted to a left Hopf algebroid (over a right bialgebroid this time) by means of the following translation map:
\begin{equation}
  \label{nasseslaub}
  \begin{array}{rcl}
    \gb_\ell^{-1} \colon V \otimes_B R &\to& (V \otimes_B R)_\bract \otimes_R  \due {(V \otimes_B R)} \lact {} ,
    \\[1mm]
v \otimes_B r &\mapsto& (r^{(-1)} v^\smam \otimes_B 1) \otimes_R (v^\smap \otimes_B r^{(0)}).
  \end{array}
  \end{equation}
Reassuming these considerations, we can state:

\begin{prop}
\label{tollhaus3}
Let $(V,B)$ be a right bialgebroid and $R$ a braided commutative monoid in $\vyd$.
Then $V \# R$ is a right bialgebroid over $R$ and it is, in addition, a left Hopf algebroid  if $(V, B)$ was a left Hopf algebroid.
\end{prop}

The typical situation to which we would like to apply this result appears in the context of duals of left (action) bialgebroids as discussed next.

\subsection{Duals of action bialgebroids}

Let $(U,A)$ be a left bialgebroid, let $\due U \lact {}$ be finitely gen\-er\-at\-ed projective as a left $A$-module, and let $\{e_j\}_{1 \leq j \leq n}  \in
U, \ \{e^j\}_{1 \leq j \leq n}  \in {U_*}$ be a dual basis for the left dual $U_* = \Hom_A(\due U \lact {}, A)$, see \S\ref{duals} in the appendix for detailed information.
To deal with the dual of $R \# U$ as a left $R$-bialgebroid, consider, to begin with,
the map
\begin{equation}
  \label{rhodia}
\eta \colon \Hom_R(R \otimes_A U, R) \to {U_*}_\bract \otimes_A R, \quad f \mapsto \Sum_j e^j \otimes_A f(1_R \otimes_A e_j),
\end{equation}
which is easily seen to be an
isomorphism of $k$-modules. As a matter of fact, it is much more:


\begin{theorem}
\label{vaexholm}
  Let $(U,A)$ be a left bialgebroid such that $\due U \lact {}$ is finitely generated $A$-projective. Then
  the map $\eta$ induces an isomorphism of right bialgebroids
  \begin{equation}
    \label{knaeckebrot}
  \Hom_R(R \# U, R) \simeq {U_*}
  \# R,
\end{equation}
which, if $U$ is a right Hopf algebroid (over a left $A$-bialgebroid),
is therefore an isomorphism of left Hopf algebroids (over right $R$-bialgebroids).
  \end{theorem}

An analogous statement can be made with respect to the right dual $U^*$, which we are not going to spell out.

\begin{proof}
  By the general theory, see Appendix \ref{duals},
the left hand side in \eqref{knaeckebrot} is a right $R$-bialgebroid since it is the dual over the base ring of a left bialgebroid;
that the right hand side is also a right $R$-bialgebroid follows from Proposition \ref{tollhaus3}
 along with Lemma~\ref{rosamunde}.

 Hence, let us show that the map $\eta$ from \eqref{rhodia} is
 not only a map of $R$-bimodules between $\due {\Hom_R(R \# U, R)} \blact \bract$ and $\due {(U_* \# R)} \blact \bract$ with respect to the notation from \eqref{soedermalm}  but also
 both a ring map as well as a coring map.
 As for being a map of $R$-bimodules, consider first:
  \begin{eqnarray*}
\eta(f \bract r)
    &
{\overset{\scriptscriptstyle{
\eqref{rhodia}
}}{=}}
&
\Sum_j  e^j \otimes_A  (f \bract r)(1_R \otimes_A e_j)
\\
&
{\overset{\scriptscriptstyle{
\eqref{duedelue}
}}{=}}
&
\Sum_j  e^j \otimes_A  f(1_R \otimes_A e_j) \cdot r
\\
&
{\overset{\scriptscriptstyle{
\eqref{narrateur}
}}{=}}
&
\big(\Sum_j  e^j \otimes_A  f(1_R \otimes_A e_j)\big)(1_{U_*} \otimes_A  r)
\\
&
{\overset{\scriptscriptstyle{
\eqref{beethovensonata23}
}}{=}}
&
\eta(f) \bract r,
  \end{eqnarray*}
for $r \in R$,  while
  \begin{eqnarray*}
\eta(r \blact f)
    &
{\overset{\scriptscriptstyle{
\eqref{rhodia}
}}{=}}
&
\Sum_j  e^j \otimes_A  (r \blact f)(1_R \otimes_A e_j)
\\
&
{\overset{\scriptscriptstyle{
\eqref{duedelue}
}}{=}}
&
\Sum_j  e^j \otimes_A  f\big((1_R \otimes_A e_j) \bract r\big)
\\
&
{\overset{\scriptscriptstyle{
\eqref{celivre}
}}{=}}
&
\Sum_j  e^j \otimes_A  f\big((1_R \otimes_A e_j)(r \otimes_A 1_U) \big)
  \end{eqnarray*}
  \begin{eqnarray*}
&
{\overset{\scriptscriptstyle{
\eqref{orangen1}
}}{=}}
&
\Sum_j  e^j \otimes_A  f({e_j}_{(1)} r \otimes_A {e_j}_{(2)} )
\\
&
{\overset{\scriptscriptstyle{
\eqref{ludovisi}
}}{=}}
&
\Sum_j  e^j \otimes_A  f\big(\langle r^{(-1)}, {e_j}_{(1)}\rangle\,r^{(0)} \otimes_A {e_j}_{(2)} \big)
\\
&
{\overset{\scriptscriptstyle{
\eqref{schizzaestrappa1}
}}{=}}
&
\Sum_{j,k}  e^j \otimes_A  f\big(\langle r^{(-1)}, {e_j}_{(1)}\rangle\,r^{(0)} \langle e^k, {e_j}_{(2)} \rangle \otimes_A e_k \big)
\\
&
{\overset{\scriptscriptstyle{
\eqref{ha2}
}}{=}}
&
\Sum_{j,k}  e^j \otimes_A  f\big(\big\langle \langle e^k, {e_j}_{(2)} \rangle \lact r^{(-1)}, {e_j}_{(1)} \big\rangle\,r^{(0)}  \otimes_A e_k \big)
\\
&
{\overset{\scriptscriptstyle{
\eqref{duedelue}
}}{=}}
&
\Sum_{j,k}  e^j \otimes_A  f\big(\big\langle  r^{(-1)}, {e_j}_{(1)} \ract \langle e^k, {e_j}_{(2)} \rangle \big\rangle\,r^{(0)}  \otimes_A e_k \big)
\\
&
{\overset{\scriptscriptstyle{
\eqref{LDMon}
}}{=}}
&
\Sum_{j,k}  e^j \otimes_A  f\big(\langle  e^k r^{(-1)}, e_j \rangle\,r^{(0)}  \otimes_A e_k \big)
\\
&
{\overset{\scriptscriptstyle{
}}{=}}
&
\Sum_{j,k}  e^j \bract \langle  e^k r^{(-1)}, e_j \rangle \otimes_A  f(r^{(0)}  \otimes_A e_k)
\\
&
{\overset{\scriptscriptstyle{
\eqref{schizzaestrappa2}
}}{=}}
&
\Sum_k  e^k r^{(-1)} \otimes_A  f(r^{(0)}  \otimes_A e_k)
\\
&
{\overset{\scriptscriptstyle{
}}{=}}
&
\Sum_k  e^k r^{(-1)} \otimes_A  r^{(0)} \cdot f(1_R  \otimes_A e_k)
\\
&
{\overset{\scriptscriptstyle{
\eqref{kindertotenlieder}
}}{=}}
&
\Sum_k  e^k r^{(-2)} \otimes_A  \big( f(1_R  \otimes_A e_k)\,r^{(-1)} \big) \cdot r^{(0)}
\\
&
{\overset{\scriptscriptstyle{
\eqref{narrateur}
}}{=}}
&
\big( \Sum_k  e^k \otimes_A  f(1_R  \otimes_A e_k)\big)(r^{(-1)} \otimes_A r^{(0)})
\\
&
{\overset{\scriptscriptstyle{
\eqref{beethovensonata23}
}}{=}}
&
r \blact \eta(f).
  \end{eqnarray*}
Therefore, $\eta$ is a morphism of $R$-bimodules, as desired.
As for $\eta$ being a ring map, let $f, g \in \Hom_R(R \# U, R) = \Hom_R(R \otimes_A U, R)$.
 The right bialgebroid multiplication of the left dual of the left bialgebroid $R \# U$ is given in \eqref{LDMon}, and hence for $u \in U$, one has
  \begin{eqnarray*}
\eta(fg)(u)
    &
{\overset{\scriptscriptstyle{
\eqref{rhodia}
}}{=}}
&
\Sum_j  \langle e^j, u \rangle \, \big( (fg)(1_R \otimes_A e_j) \big)
\\
&
{\overset{\scriptscriptstyle{
\eqref{LDMon}
}}{=}}
&
\Sum_j  \langle e^j, u \rangle \, g\big((1_R \otimes_A e_j)_{(1)} \ract f\big((1_R \otimes_A e_j)_{(2)}\big) \big)
\\
&
{\overset{\scriptscriptstyle{
\eqref{mandarinen1}
}}{=}}
&
\Sum_j  \langle e^j, u \rangle \, g\pig(t \big(f(1_R \otimes_A {e_j}_{(2)})\big) (1_R \otimes_A {e_j}_{(1)}) \pig)
\\
&
{\overset{\scriptscriptstyle{
\eqref{celivre}, \eqref{orangen1}
}}{=}}
&
\Sum_j  \langle e^j, u \rangle \, g\big(f(1_R \otimes_A {e_j}_{(2)})_{[0]} \otimes_A f(1_R \otimes_A {e_j}_{(2)})_{[1]} {e_j}_{(1)} \big)
\\
&
{\overset{\scriptscriptstyle{
}}{=}}
&
\Sum_j  g \pig(\langle e^j, u \rangle \, f(1_R \otimes_A {e_j}_{(2)})_{[0]} \otimes_A f(1_R \otimes_A {e_j}_{(2)})_{[1]} {e_j}_{(1)} \pig)
\\
&
{\overset{\scriptscriptstyle{
\eqref{ha3}
}}{=}}
&
\Sum_j  g \pig(f(1_R \otimes_A {e_j}_{(2)})_{[0]} \otimes_A f(1_R \otimes_A {e_j}_{(2)})_{[1]} s(\langle e^j, u \rangle) {e_j}_{(1)} \pig)
\\
&
{\overset{\scriptscriptstyle{
\eqref{schizzaestrappa1}
}}{=}}
&
g \pig(f(1_R \otimes_A u_{(2)})_{[0]} \otimes_A f(1_R \otimes_A u_{(2)})_{[1]} u_{(1)} \pig).
  \end{eqnarray*}
On the other hand,
  \begin{eqnarray*}
&&
    \big(\eta(f)\eta(g)\big)(u)
    \\
    &
{\overset{\scriptscriptstyle{
\eqref{rhodia}
}}{=}}
&
\Sum_{j,k}  \pig(\big( e^j \otimes_A f(1_R \otimes_A e_j) \big)\big( e^k \otimes_A g(1_R \otimes_A e_k) \big)   \pig)(u)
\\
  &
{\overset{\scriptscriptstyle{
\eqref{narrateur}
}}{=}}
&
\Sum_{j,k}
\langle e^j {e^k}^{(1)}, u \rangle \, \pig(\big(f(1_R \otimes_A e_j){e^k}^{(2)}\big) \cdot g(1_R \otimes_A e_k) \pig)
\\
  &
{\overset{\scriptscriptstyle{
\eqref{onirique}
}}{=}}
&
\Sum_{j,k}
\langle e^j {e^k}^{(1)}, u \rangle \,
\pig(
\big(f(1_R \otimes_A e_j)_{[0]} \langle {e^k}^{(2)}, f(1_R \otimes_A e_j)_{[1]} \rangle \big)
\cdot g(1_R \otimes_A e_k)
\pig)
\\
  &
{\overset{\scriptscriptstyle{
\eqref{LDMon}
}}{=}}
&
\Sum_{j,k}
\big\langle {e^k}^{(1)}, u_{(1)} \ract \langle e^j, u_{(2)} \rangle \big\rangle \,
\pig(
\big(f(1_R \otimes_A e_j)_{[0]} \, \langle {e^k}^{(2)}, f(1_R \otimes_A e_j)_{[1]} \rangle \big)
\cdot g(1_R \otimes_A e_k)
\pig)
\\
&
{\overset{\scriptscriptstyle{
\eqref{marmor2}
}}{=}}
&
\Sum_{j,k}
\pig(\big\langle {e^k}^{(1)}, u_{(1)} \ract \langle e^j, u_{(2)} \rangle \big\rangle \,
\big(f(1_R \otimes_A e_j)_{[0]} \, \langle {e^k}^{(2)}, f(1_R \otimes_A e_j)_{[1]} \rangle \big) \pig)
\cdot g(1_R \otimes_A e_k)
\\
  &
{\overset{\scriptscriptstyle{
\eqref{ha3}
}}{=}}
&
\Sum_{j,k}  \pig(f(1_R \otimes_A e_j)_{[0]} \, \Big\langle {e^k}^{(2)}, f(1_R \otimes_A e_j)_{[1]}
\bract \big\langle {e^k}^{(1)}, u_{(1)} \ract \langle e^j, u_{(2)} \rangle \big\rangle
\Big\rangle \pig)
\cdot g(1_R \otimes_A e_k)
\\
&
{\overset{\scriptscriptstyle{
\eqref{trattovideo}
}}{=}}
&
\Sum_{j,k}  \big(f(1_R \otimes_A e_j)_{[0]} \, \big\langle e^k, f(1_R \otimes_A e_j)_{[1]}
t(\langle e^j, u_{(2)} \rangle) u_{(1)}  \big\rangle
\big)
\cdot g(1_R \otimes_A e_k)
\\
&
{\overset{\scriptscriptstyle{
\eqref{ha3}
}}{=}}
&
\Sum_{j,k}
\pig(\big(\langle e^j, u_{(2)} \rangle \, f(1_R \otimes_A e_j)\big)_{[0]} \, \big\langle e^k, \big(\langle e^j, u_{(2)} \rangle \, f(1_R \otimes_A e_j)\big)_{[1]} u_{(1)}  \big\rangle
\pig)
\cdot g(1_R \otimes_A e_k)
\\
&
{\overset{\scriptscriptstyle{
}}{=}}
&
\Sum_{k}
\big(f(1_R \otimes_A u_{(2)})_{[0]} \, \langle e^k, f(1_R \otimes_A u_{(2)})_{[1]} u_{(1)} \rangle
\big)
\cdot g(1_R \otimes_A e_k)
\\
&
{\overset{\scriptscriptstyle{
\eqref{marmor2}  
}}{=}}
&
f(1_R \otimes_A u_{(2)})_{[0]} \cdot g(1_R \otimes_A f(1_R \otimes_A u_{(2)})_{[1]})
\\
&
{\overset{\scriptscriptstyle{
}}{=}}
&
g\pig(f(1_R \otimes_A u_{(2)})_{[0]} \otimes_A f(1_R \otimes_A u_{(2)})_{[1]}\pig).
 \end{eqnarray*}
  Hence, $\eta(fg) = \eta(f)\eta(g)$, and so $\eta$ is an isomorphism of $\Ree$-rings. As for being an $R$-coring morphism, consider first, similar to \eqref{zuschlag}, the following isomorphism:
\begin{small}
  \begin{equation}
\label{zuschlag1}
\begin{array}{rcl}
  ({U_*} \otimes_A R)_\bract \otimes_R  \due {({{U_*}} \otimes_A R)} \blact {}
  &\to&
  \Hom_A(U_\bract \otimes_A \due U \lact {}, R),
  \\
  (\psi \otimes_A r)  \otimes_R  (\tilde \psi \otimes_A r')
  &\mapsto&
\big  \{ u  \otimes_A u'
  \mapsto
\big\langle \tilde \psi r^{(-1)}, u \bract \langle \psi , u' \rangle \big\rangle (r^{(0)} \cdot r')
 \big \},
\end{array}
  \end{equation}
  \end{small}based on the observation that
  \begin{equation*}
    \begin{split}
  (\psi \otimes_A r)  \otimes_R  (\tilde \psi \otimes_A r')
  &=
      (\psi \otimes_A 1_R)  \otimes_R  \big(\tilde \psi r^{(-2)} \otimes_A (r' r^{(-1)}) \cdot r^{(0)} \big)
      \\
    &=  (\psi \otimes_A 1_R)  \otimes_R  (\tilde \psi r^{(-1)} \otimes_A r^{(0)} \cdot r'),
    \end{split}
    \end{equation*}
  using the tensor product over $R$
  by means of the target map \eqref{beethovensonata23}, the product \eqref{narrateur}, as well as the braided commutativity \eqref{kindertotenlieder}. This isomorphism induces a nondegenerate pairing, and we can write
    \begin{eqnarray}
\nonumber
    &&  \big\langle (\psi \otimes_A r) \otimes_R (\tilde \psi \otimes_A r'),
u \otimes_A u'
      \big\rangle
      \\
      \nonumber
      &
{\overset{\scriptscriptstyle{
\eqref{zuschlag1}, \eqref{marmor2}
}}{=}}
&
\big\langle \tilde \psi r^{(-1)},
u \bract \langle \psi , u' \rangle
\big\rangle
\, r^{(0)} \cdot r'
\\
\nonumber
&
{\overset{\scriptscriptstyle{
\eqref{LDMon}
}}{=}}
&
\big\langle r^{(-1)},
(u_{(1)} \bract \langle \psi , u' \rangle) \ract
\langle \tilde \psi, u_{(2)} \rangle
\big\rangle
\, r^{(0)} \cdot r'
\\
\nonumber
&
{\overset{\scriptscriptstyle{
\eqref{duedelue}
}}{=}}
&
\big\langle \langle \tilde \psi, u_{(2)} \rangle \lact r^{(-1)},
u_{(1)} \bract \langle \psi , u' \rangle
\big\rangle
\, r^{(0)} \cdot r'
\\
\nonumber
&
{\overset{\scriptscriptstyle{
\eqref{ha2}
}}{=}}
&
\big\langle r^{(-1)},
u_{(1)} \bract \langle \psi , u' \rangle
\big\rangle
\, r^{(0)}\cdot  \langle \tilde \psi, u_{(2)} \rangle \, r'
\\
\label{casinoaurora}
&
{\overset{\scriptscriptstyle{
\eqref{ludovisi}
}}{=}}
&
\big((u_{(1)} \bract \langle \psi , u' \rangle)
r\big) \cdot \langle \tilde \psi, u_{(2)} \rangle \, r',
 \end{eqnarray}
where in the last equation an element of $U$ acts on $r$ from the left. Observe that
    $(ar) \cdot r' = (ar) \cdot r'$ and $(ra) \cdot r' = r \cdot (ar')$ as in \eqref{marmor2}, which permits to avoid writing a couple of parentheses both above and below.
If $\gD_r$, by abuse of notation, denotes both the (right bialgebroid) coproduct in $\Hom_R(R \# U, R)$ as well as in $U_* \# R$, we compute for $f \in \Hom_R(R \# U, R) = \Hom_R(R \otimes_A U, R)$: 
    \begin{eqnarray*}
&&
      \big\langle (\eta \otimes_R \eta) \circ \gD_r f, u \otimes_A u' \big\rangle
      \\
      &=&
\big\langle \eta f^{(1)} \otimes_R \eta f^{(2)}, u \otimes_A u' \big\rangle
      \\
      &
{\overset{\scriptscriptstyle{
\eqref{rhodia}
}}{=}}
&
\Sum_{i, j}
\big\langle
\big( e^j \otimes_A f^{(1)}(1_R \otimes_A e_j)\big)
\otimes_R
\big( e^i \otimes_A f^{(2)}(1_R \otimes_A e_i)\big),
u \otimes_A u'
\big\rangle
\\
&
{\overset{\scriptscriptstyle{
\eqref{casinoaurora}
}}{=}}
&
\Sum_{i, j}
\big(( u_{(1)} \bract \langle e^j, u' \rangle )f^{(1)}(1_R \otimes_A e_j)\big)
\cdot
\langle e^i, u_{(2)} \rangle f^{(2)}(1_R \otimes_A e_i)
\\
&
{\overset{\scriptscriptstyle{
}}{=}}
&
\big(u_{(1)} f^{(1)}(1_R \otimes_A u')\big)
\cdot f^{(2)}(1_R \otimes_A u_{(2)})
\\
&
{\overset{\scriptscriptstyle{
}}{=}}
&
f^{(2)}\pig(\big(u_{(1)} f^{(1)}(1_R \otimes_A u')\big) \otimes_A u_{(2)}\pig),
    \end{eqnarray*}
    where we used $A$-linearity and $R$-linearity in the last two steps.
     On the other hand,
    \begin{eqnarray*}
      \big\langle \gD_r (\eta f), u \otimes_A u' \big\rangle
      &
{\overset{\scriptscriptstyle{
\eqref{mandarinen2}
}}{=}}
&
\Sum_{j}
\big\langle \big( {e^j}^{(1)} \otimes_A 1_R  \big)
\otimes_R \big( {e^j}^{(2)} \otimes_A f(1_R \otimes_A e_j)\big) , u \otimes_A u' \big\rangle
 \\
      &
{\overset{\scriptscriptstyle{
\eqref{casinoaurora}
}}{=}}
&
\Sum_{j}
\big( (u_{(1)} \bract \langle {e^j}^{(1)}, u' \rangle) 1_R \big)
\cdot
\langle {e^j}^{(2)}, u_{(2)} \rangle \, f(1_R \otimes_A e_j)
 \\
      &
{\overset{\scriptscriptstyle{
\eqref{marmor1}
}}{=}}
&
\Sum_{j}
\gve\big(u_{(1)} \bract \langle {e^j}^{(1)}, u' \rangle\big) \langle {e^j}^{(2)}, u_{(2)} \rangle \, f(1_R \otimes_A e_j)
 \\
      &
{\overset{\scriptscriptstyle{
}}{=}}
&
\Sum_{j}
 f\pig (1_R \otimes_A \big\langle {e^j}^{(2)}, u \bract \langle {e^j}^{(1)}, u' \rangle \big\rangle \, e_j\pig)
 \\
      &
{\overset{\scriptscriptstyle{
\eqref{trattovideo}
}}{=}}
&
\Sum_{j}
 f\big (1_R \otimes_A \langle {e^j}, u u' \rangle \, e_j\big)
 \\
      &
{\overset{\scriptscriptstyle{
}}{=}}
&
 f (1_R \otimes_A u u')
 \\
      &
{\overset{\scriptscriptstyle{
\eqref{orangen1}
}}{=}}
&
 f \big((1_R \otimes_A u)((1_R \otimes_A u')\big)
 \\
      &
{\overset{\scriptscriptstyle{
\eqref{trattovideo}
}}{=}}
&
f^{(2)} \big(
(1_R \otimes_A u) \bract f^{(1)}(1 \otimes_A u')
\big)
\\
      &
{\overset{\scriptscriptstyle{
\eqref{celivre}
}}{=}}
&
f^{(2)} \pig(
\big(1_R \otimes_A u\big)\big(f^{(1)}(1 \otimes_A u') \otimes_A 1_U \big)
\pig)
\\
&
{\overset{\scriptscriptstyle{
\eqref{orangen1}
}}{=}}
&
f^{(2)}\pig(\big(u_{(1)} f^{(1)}(1_R \otimes_A u')\big) \otimes_A u_{(2)}\pig),
    \end{eqnarray*}
where we used left $A$-linearity twice in step four, and the coproduct on the dual of the left bialgebroid $R \# U$ analogous to \eqref{trattovideo} in step eight.
 This is now the same as above and therefore $\eta$ is a map of $R$-corings, as claimed.

 Finally, to show that $\eta$ is also an isomorphism of left Hopf algebroids (over underlying right $R$-bialgebroids), this follows automatically by the very definition of a left Hopf algebroid: according to Eq.~\eqref{ibianchi2}, this amounts to the invertibility of the map $\gb_\ell$, which only uses the product and the coproduct, which already have been shown to be compatible with $\eta$. Nevertheless, although redundant, a direct proof can be given which we consider to be instructive as it uses the explicit expression of the translation map \eqref{staendigallergie1} for the dual, that is, the fact that the dual of a right Hopf algebroid (over a left bialgebroid) is a left Hopf algebroid (over a right bialgebroid).
 In this spirit, consider, analogous to \eqref{zuschlag1},
 the isomorphism:
\begin{small}
 \begin{equation}
\label{zuschlag2}
\begin{array}{rcl}
  ({U_*} \otimes_A R)_\bract \otimes_R  \due {({{U_*}}\otimes_A R)} \lact {}
  &\to&
  \Hom_A(U_\ract \otimes_A \due U \lact {}, R),
  \\[1mm]
  (\psi \otimes_A r)  \otimes_R  (\tilde \psi \otimes_A r')
  &\mapsto&
 \big \{ u  \otimes_A v
  \mapsto
  \big\langle
  \tilde \psi^{(1)}, u \ract \langle \psi , u' \rangle
  \big\rangle
\big(  (r \tilde \psi^{(2)} )\cdot r' \big)
\big \},
\end{array}
  \end{equation}
  \end{small}based here on the observation that
  \begin{equation*}
    \begin{split}
(\psi \otimes_A r)  \otimes_R  (\tilde \psi \otimes_A r')
 &=
 (\psi \otimes_A 1_R)  \otimes_R  \big((1 \otimes_A r)(\tilde \psi \otimes_A r') \big)
  \\
 &=  (\psi \otimes_A 1_R)  \otimes_R  \big(\tilde \psi^{(1)} \otimes_A (r \tilde \psi^{(2)}) \cdot r'\big),
    \end{split}
    \end{equation*}
 using the tensor product over $R$
 by means of the source map \eqref{beethovensonata23} and the product \eqref{narrateur}.
  This isomorphism induces again a nondegenerate pairing, and this time we can write
    \begin{eqnarray}
      \nonumber
      &&
      \big\langle (\psi \otimes_A r) \otimes_R (\tilde \psi \otimes_A r'),
u \otimes_A u'
      \big\rangle
      \\
      \nonumber
      &
{\overset{\scriptscriptstyle{
\eqref{zuschlag2}, \eqref{marmor2}
}}{=}}
&
 \big\langle
  \tilde \psi^{(1)}, u \ract \langle \psi , u' \rangle
  \big\rangle
(r \tilde \psi^{(2)} )\cdot r'
\\
\nonumber
 &
 {\overset{\scriptscriptstyle{
\eqref{onirique}
}}{=}}
 &
\big\langle
  \tilde \psi^{(1)}, u  \ract \langle \psi , u' \rangle
  \big\rangle
\big( r_{[0]} \big\langle \tilde \psi^{(2)} , r_{[1]} \big\rangle \big)\cdot r'
\\
\nonumber
 &
 {\overset{\scriptscriptstyle{
\eqref{marmor2},
       \eqref{ha3}
}}{=}}
 &
\pig( r_{[0]} \big\langle \tilde \psi^{(2)} , r_{[1]} \bract \big\langle
  \tilde \psi^{(1)}, u  \ract \langle \psi , u' \rangle
  \big\rangle \big\rangle \pig)\cdot r'
\\
\nonumber
 &
 {\overset{\scriptscriptstyle{
\eqref{LDMon}
}}{=}}
 &
\pig( r_{[0]} \big\langle \tilde \psi, r_{[1]} (u  \ract \langle \psi , u' \rangle)
  \big\rangle \pig)\cdot r'
  \\
  \label{casinoaurora2}
   &
 {\overset{\scriptscriptstyle{
\eqref{marmor2}
}}{=}}
 &
\pig( r_{[0]} \big\langle \tilde \psi, r_{[1]} u  \ract \langle \psi , u' \rangle
  \big\rangle \pig)\cdot r'.
\end{eqnarray}
  With this, let us compute for $f \in \Hom_R(R \# U, R) = \Hom_R(R \otimes_A U, R)$:
    \begin{eqnarray*}
&&
      \big\langle (\eta \otimes_R \eta) \circ \gb_\ell^{-1}(f), u \otimes_A u' \big\rangle
      \\
      &=&
\big\langle \eta f^\smam \otimes_R \eta f^\smap, u \otimes_A u' \big\rangle
      \\
      &
{\overset{\scriptscriptstyle{
\eqref{rhodia}
}}{=}}
&
\Sum_{i, j}
\big\langle
\big( e^j \otimes_A f^\smam(1_R \otimes_A e_j)\big)
\otimes_R
\big( e^i \otimes_A f^\smap(1_R \otimes_A e_i)\big),
u \otimes_A u'
\big\rangle
      \\
      &
{\overset{\scriptscriptstyle{
\eqref{casinoaurora2}
}}{=}}
&
\Sum_{i, j}
f^\smam(1_R \otimes_A e_j)_{[0]} \big \langle e^i, f^\smam(1_R \otimes_A e_j)_{[1]}
u \ract \langle e^j, u' \rangle
\big\rangle \cdot f^\smap(1_R \otimes_A e_i)
  \\
      &
{\overset{\scriptscriptstyle{
\eqref{ha3}, \eqref{marmor1}
}}{=}}
&
f^\smam(1_R \otimes_A u')_{[0]} \cdot f^\smap(1_R \otimes_A f^\smam(1_R \otimes_A u')_{[1]} u),
  \\
      &
{\overset{\scriptscriptstyle{
}}{=}}
&
f^\smap(f^\smam(1_R \otimes_A u')_{[0]}  \otimes_A f^\smam(1_R \otimes_A u')_{[1]} u),
  \\
      &
{\overset{\scriptscriptstyle{
\eqref{celivre}, \eqref{orangen1}
}}{=}}
&
f^\smap\big((1_R \otimes_A
u) \ract f^\smam(1_R \otimes_A u')\big),
    \end{eqnarray*}
  where we also used $R$-linearity in the last two steps. On the other hand,
  \begin{eqnarray*}
    &&
      \big\langle \gb_\ell^{-1} (\eta f), u \otimes_A u' \big\rangle
\\
      &
{\overset{\scriptscriptstyle{
\eqref{rhodia}
}}{=}}
&
\Sum_{j}
\big\langle \beta_\ell^{-1} \big( e^j \otimes_A f(1_R \otimes_A e_j)\big) , u \otimes_A u' \big\rangle
 \\
      &
{\overset{\scriptscriptstyle{
\eqref{nasseslaub}
}}{=}}
&
\Sum_{j}
\big\langle \big( f(1_R \otimes_A e_j)^{(-1)} {e^j}^\smam \otimes_A 1_R \big)
\otimes_R
\big( {e^j}^\smap \otimes_A f(1_R \otimes_A e_j)^{(0)}\big)
, u \otimes_A u' \big\rangle
 \\
      &
{\overset{\scriptscriptstyle{
\eqref{casinoaurora2}
}}{=}}
&
\Sum_{j}
\big\langle {e^j}^\smap, u \ract
\langle
f(1_R \otimes_A e_j)^{(-1)} {e^j}^\smam, u' \rangle
\big\rangle \cdot f(1_R \otimes_A e_j)^{(0)}
 \\
      &
{\overset{\scriptscriptstyle{
\eqref{LDMon}
}}{=}}
&
\Sum_{j}
\pig\langle {e^j}^\smap
, u \ract
\big\langle
{e^j}^\smam, u'_{(1)} \ract \langle
f(1_R \otimes_A e_j)^{(-1)} , u'_{(2)} \rangle \big\rangle
\pig\rangle \cdot f(1_R \otimes_A e_j)^{(0)}
 \\
      &
{\overset{\scriptscriptstyle{
\eqref{Uch1}
}}{=}}
&
\Sum_{j}
\big\langle {e^j}^\smap \bract  \langle
f(1_R \otimes_A e_j)^{(-1)} , u'_{(2)} \rangle
, u \ract
\langle
{e^j}^\smam, u'_{(1)}  \rangle
\big\rangle \cdot f(1_R \otimes_A e_j)^{(0)}
 \\
      &
{\overset{\scriptscriptstyle{
\eqref{duedelue}, \eqref{marmor2}
}}{=}}
&
\Sum_{j}
\big\langle {e^j}^\smap , u \ract
\langle
{e^j}^\smam, u'_{(1)}  \rangle
\big\rangle \cdot  \big( \langle
f(1_R \otimes_A e_j)^{(-1)} , u'_{(2)} \rangle
 f(1_R \otimes_A e_j)^{(0)} \big)
 \\
      &
{\overset{\scriptscriptstyle{
\eqref{ludovisi}
}}{=}}
&
\Sum_{j}
\big\langle {e^j}^\smap , u \ract
\langle
{e^j}^\smam, u'_{(1)}  \rangle
\big\rangle \cdot  \big( u'_{(2)}
 f(1_R \otimes_A e_j) \big)
 \\
      &
{\overset{\scriptscriptstyle{
\eqref{duedelue}
}}{=}}
&
\Sum_{j}
\big\langle
\langle
{e^j}^\smam, u'_{(1)}  \rangle \lact {e^j}^\smap , u
\big\rangle \cdot  \big( u'_{(2)}
 f(1_R \otimes_A e_j) \big)
 \\
      &
{\overset{\scriptscriptstyle{
\eqref{staendigallergie3}
}}{=}}
&
\Sum_{j}
\langle
u'_{(1)}  \manda {e^j}, u
\rangle \cdot  \big( u'_{(2)}
 f(1_R \otimes_A e_j) \big)
  \end{eqnarray*}
  \begin{eqnarray*}
      &
{\overset{\scriptscriptstyle{
\eqref{staendigallergie2}
}}{=}}
&
\Sum_{j}
\gve\big( u'_{(1)\smap}  \bract \langle {e^j}, u'_{(1)\smam} u \rangle
\big) \cdot  \big( u'_{(2)}
 f(1_R \otimes_A e_j) \big)
\\
      &
{\overset{\scriptscriptstyle{
\eqref{Tch4}
}}{=}}
&
\Sum_{j}
\gve\big( u'_{\smap(1)}  \bract \langle {e^j}, u'_{\smam} u \rangle
\big) \cdot  \big( u'_{\smap (2)}
 f(1_R \otimes_A e_j) \big)
 \\
      &
{\overset{\scriptscriptstyle{
}}{=}}
&
u'_{\smap}
 f(1_R \otimes_A u'_{\smam} u )
 \\
      &
{\overset{\scriptscriptstyle{
\eqref{orangen1}
}}{=}}
&
u'_{\smap}
 f\big((1_R \otimes_A u'_{\smam})(1_R \otimes_A u ) \big)
 \\
      &
{\overset{\scriptscriptstyle{
\eqref{sicherheitshinweis1}, \eqref{erkaeltet}
}}{=}}
&
\gve_{R \# U}\pig((1_R \otimes_A u')_{\smap}
\bract
f\big((1_R \otimes_A u')_{\smam} (1_R \otimes_A u ) \big)
 \\
      &
{\overset{\scriptscriptstyle{
\eqref{staendigallergie2}
}}{=}}
&
\big((1_R \otimes_A u') \manda f\big)(1_R \otimes_A u )
 \\
      &
{\overset{\scriptscriptstyle{
\eqref{staendigallergie3}
}}{=}}
&
\big( \langle f^\smam, 1_R \otimes_A u' \rangle
 \lact f^\smap\big) (1_R \otimes_A u)
 \\
      &
{\overset{\scriptscriptstyle{
\eqref{duedelue}
}}{=}}
&
f^\smap
 \big(  (1_R \otimes_A u) \ract f^\smam  (1_R \otimes_A u')\big),
  \end{eqnarray*}
  which is the same as above, and by which one concludes the proof that $\eta$ is also a morphism of Hopf structures in the sense explained above.
\end{proof}

\section{Quantum duality for quantum action groupoids}
\label{four}
 In this section, we are going to deal with topological left (or right) bialgebroids, with respect to some suitable adic topologies.  The main objects of interest are the so-called {\em quantum groupoids}, forming two special categories (roughly, quantisations of Lie groupoids and of Lie-Rinehart algebras) that are treated to some extent in  \cite{CheGav:DFFQG}.  In particular, these behave well with respect to linear duality (which yields antiequivalences between categories of the two kinds); also, they feature a special phenomenon, known as  {\em quantum duality principle}, which establishes equivalences between these categories via the so-called {\em Drinfeld functors}.  In the present section,
 we shall show how Drinfeld functors apply to quantum groupoids that are also action bialgebroids (by applying our previous results on linear duality), finding that Drinfeld functors commute with the action bialgebroid construction.

\subsection{Quantum groupoids and the Quantum Duality Principle.}  \label{intr-QDP}
In this subsection, we introduce the notion of {\em quantum groupoids}: these are special {\em quantum} bialgebroids, that is, (topological) bialgebroids which,
for a   Lie-Rinehart algebra $(L,A)$,
are formal deformations of the universal enveloping algebra $  V\!L  $ as a left $A$-bialgebroid or  its $A$-linear dual $ J\!L  $, the jet space, as a right $A$-bialgebroid.
Afterwards, we will recall linear duality and the quantum duality principle for these quantum groupoids; all necessary details can be found in  \cite{CheGav:DFFQG}.

 \begin{definition}
\
   \begin{enumerate}
     \compactlist{99}
     \item
       A  {\em left quantum universal enveloping bialgebroid}
       is given by a topological left bialgebroid  $( U_h, A_h, s_h, t_h, \Delta_h, \gve_h)$
over a topological  $ k[[h]] $-algebra  $ A_h $
              with respect to  the  $ h $-adic  topology
    such that:
\begin{enumerate}
  \compactlist{99}
\item
  $ A_h \simeq A[[h]] $
 as topological  $k[[h]]$-modules
  such that $A_h/hA_h \simeq A$
as  $ k $-algebras,
 for some commutative $k$-algebra  $A$;
\item
$U_h \simeq V\!L  [[h]]$
  as topological  $k[[h]] $-modules
with respect to  the  $ h $-adic  topology
  for some finitely generated
 projective Lie-Rinehart algebra $(L,A)$;
 \item
   $ \big(U_h /  h U_h, A_h/hA_h\big)  \simeq  \big(V\!L [[h]]/hV\!L [[h]], A\big)$  as left bialgebroids over $A$;
   \item
   the coproduct $\gD_h$ takes values in the completion
  \begin{equation*}
    \begin{split}
      \qquad\qquad\quad
      U_{h \ract}  \widehat{\times}_{A_h}  \due {U_h} \lact {}
:=
    \big\{ \Sum_i u_ i \otimes u'_i \in U_{h \ract} \, \widehat{\otimes}_{A_h} \, \due {U_h} \lact {} \mid \Sum_i a \blact u_i \otimes u'_i = \Sum_i u_i \otimes  u'_i \bract a \big\}
  \end{split}
  \end{equation*}
  of the Takeuchi product,
    where  $ U_{h \ract } \widehat{\otimes}_{A_h} \due {U_h} \lact {} $  denotes the  $ h $-adic  completion of the tensor product  $ U_{h \ract } \otimes_{A_h} \due {U_h} \lact {}. $
\end{enumerate}
In such a setting, we call  $ U_h $  a  {\em quantisation} or a  {\em quantum deformation}  of  $ V\!L  $,
a situation we denote by $V_hL := U_h$.
                                                          \item
   A  {\em right quantum formal series bialgebroid}  is given by a topological right bialgebroid  $( F_h, A_h, s_h, t_h, \Delta_h, \pl_h)$  over a topological  $ k[[h]] $-algebra  $ A_h $
  with respect to  the  $ h $-adic  topology
   such that:
\begin{enumerate}
  \compactlist{99}
\item
   same condition as (a) in (i)  above,  for some commutative $k$-algebra  $A$;
\item
$F_h \simeq J\!L [[h]]$
  as topological  $k[[h]] $-modules
with respect to  the  $ h $-adic  topology
  for some finitely generated
  projective Lie-Rinehart algebra $(L,A)$;
\item
  the topology of  $ F_h $  is the  $ I_h $-adic  one, where
  \begin{equation*}
  I_h := \{\,f \in F_h \mid \pl_h(f) \in h F_h\,\},
  \end{equation*}
  and  $ F_h $  is  $ I_h $-adically  complete;
\item
 $ \big(F_h /  h F_h, A_h/hA_h\big) \simeq \big(J\!L[[h]],hJ\!L[[h]]\big)$  as right bialgebroids over  $ A $;
\item
   the coproduct $\gD_h$ takes values in the completion
  \begin{equation*}
    \begin{split}
      \qquad\qquad\quad
      F_{h \bract}  \widetilde{\times}_{A_h}  \due {F_h} \blact {}
:=
    \big\{ \Sum_i f_ i \otimes f'_i \in F_{h \bract} \, \widetilde{\otimes}_{A_h} \, \due {F_h} \blact {} \mid \Sum_i a \lact f_i \otimes f'_i = \Sum_i f_i \otimes  f'_i \ract a \big\}
  \end{split}
  \end{equation*}
  of the Takeuchi product,
    where  $ F_{h \bract } \widetilde{\otimes}_{A_h} \due {F_h} \blact {} $  denotes the completion of $ F_{h \bract } \otimes_{A_h} \due {F_h} \blact {} $ with respect to the topology defined by the filtration $\big\{ \Sum_{p+q=n} I^p_h \otimes_{A_h} I^q_h \big\}$.
\end{enumerate}
In such a setting, we call  $ F_h $  a  {\em quantisation} or a  {\em quantum deformation}  of  $ J\!L $,
 a situation that we denote by $J_hL := F_h$.
   \end{enumerate}
\end{definition}

 In a parallel way, one can easily formulate the respective definitions for {\em right} quantum universal enveloping bialgebroids resp.\ {\em left} quantum formal series bialgebroid.

 Clearly, quantum bialgebroids of all four types form obvious categories when defining morphisms as is done for (topological) bialgebroids in general: for some fixed  $ k[[h]] $-algebra  $ A_h   $, these four categories
will be denoted by
$ \lqua, \rqua, \lqfa$, and $\rqfa $, respectively.
See \cite[\S4.1]{CheGav:DFFQG}  for further details.

   \begin{rem}
The definition above makes sense also if one drops the finite-projectiveness assumption; nevertheless, all results we are interested in (starting from those concerning linear duality) only work well if that assumption is included.
     \end{rem}

\begin{rem}
  \label{rmk: Poisson-from-quantum}
 It is worth stressing that if any (left or right) quantum universal enveloping bialgebroid  $ V_hL $  is given, then  $ (L,A) $  inherits from that a Lie cobracket which turns it into a Lie-Rinehart bialgebra.  Similarly, if any (right or left) quantum formal series bialgebroid  $ J_hL $  is given, then it induces on  $ (L,A) $  a Lie cobracket which makes it into a Lie-Rinehart bialgebra.  
 Note that the assumption of  $ L $  being finite projective here ensures that the $A$-linear dual  $ L^* $  is a Lie-Rinehart bialgebra as well; on the other hand, it is also necessary in order to have a well defined bialgebroid structure on  $ J\!L $, see \cite[\S4.1]{CheGav:DFFQG}  again.
\end{rem}

\subsection{Linear duality beyond projective finiteness}
\label{limits}
If  $ (U,A) $  is a left bialgebroid such that   $ \due U \lact {} $  is finitely generated  $A$-projective,  then its left dual  $  U_* = \Hom_A( \due U \lact {} ,  A)  $ is a right $A$-bialgebroid, see Appendix \ref{duals}. As follows from \eqref{schizzaestrappa2}, this is again finitely generated projective over  $ A $  for the right $ A $-module  structure ${U_*}_\bract$. Therefore, one can dualise again, and similar considerations hold for the right dual $U^*$ or when starting from a right bialgebroid. Dualising twice, the final outcome is (canonically isomorphic to) the original bialgebroid  $ U $,  for example,
$  {}^*(U_*) = U  $, and likewise for all other choices.

More interesting, the finite-projectiveness assumption can be relaxed:
if, for example,  $ \due U \lact {} $  is just a {\em direct} limit of some directed system of finitely generated projective  $ A $-modules, that is,
if it is
{\em ind}-finitely generated projective,
then $U_*$ admits a canonical structure of a right bialgebroid over  $ A  $,  given by the same construction as in the finite case, with the difference that $ U_* $  is a
{\em topological}  bialgebroid whose underlying topology is the weak one, and accordingly the coproduct now takes values in a suitable Takeuchi product of  $ U_* $  with itself. In particular, $  U_* $  is an {\em inverse} limit of finite projectives, that is, it is {\em pro}-finitely generated projective, and therefore one can consider its (left and right) {\em continuous}  dual, denoted  $  {}_\star(U_*)  $  and  $  {}^\star(U_*)  $,  where {\em continuous} refers, as hinted at, to the weak topology on  $ U_* $  and the trivial topology on  $ A  $. This is tailored in such a way that, for example,
$  {}^\star(U_*) = U  $, that is, linear dualisation composed with continuous dualisation, yields the identity functor. Similar considerations can be made with respect to the right dual $U^*$.
In this section, we will apply these ideas to left or right quantum universal bialgebroids and left or right quantum formal series bialgebroids.

Let $(L,A)$ be a Lie Rinehart algebra and let us always assume that $L$ is a finitely generated projective $A$-module.
A quantum universal enveloping bialgebroid $V_h L$, endowed with the $h$-adic topology,  is a topological bialgebroid. From the fact that  $ V\!L  $  is inv-finite projective (over  $ A  $),  one deduces that any quantisation  $ V_h L $  is inv-finite projective as well   (over  $ A_h  $).
On the other hand, let $J_h L$ be a right quantum formal series bialgebroid  with  counit denoted  by $\partial_h$, and set $I_h := \partial^{-1}_h(hA_h)$.
The weak topology on $J_h L$ is the $I_h$-adic topology.
  Similarly, since  $ J\!L $  is pro-finite projective (over~$ A  $),  any quantisation  $ J_hL $  is pro-finite projective as well (over~$ A_h  $). The same remarks hold for left quantum formal series bialgebroids.

 Denote by ${}^\star(-)$, ${}_\star(-)$, $(-)^\star$, and $(-)_\star $ the respective subspaces of  $  {}^*(-)$, ${}_*(-)$, $(-)^*$, and $(-)_* $,  consisting of maps that are continuous when $J_h L $ is endowed with the $I_{J_h L}$-adic topology and $A_h$ is endowed with the $h$-adic topology.
  Considering linear duality, we can state \cite[\S5]{CheGav:DFFQG}:

 \begin{theorem}
    Left and right duals yield well-defined  contravariant  functors
    \begin{equation*}
      \begin{array}{ll}
        ^{\phantom{*}}(-)_* \colon\, \lqua \to \rqfa, \quad\qquad & {}^\star(-)^{\phantom{\star}} \colon\, \rqfa \to \lqua,
        \\[1mm]
        ^{\phantom{*}}(-)^* \colon\, \lqua \to \rqfa, \quad\qquad & {}_\star(-)^{\phantom{\star}} \colon\, \rqfa \to \lqua,
        \\[1mm]
          {}_*(-)^{\phantom{*}} \colon\, \rqua \to \lqfa, \quad\qquad &
          ^{\phantom{\star}}(-)^\star \colon\, \lqfa \to \rqua,
        \\[1mm]
          {}^*(-)^{\phantom{*}} \colon\, \rqua \to \lqfa, \quad\qquad &
          ^{\phantom{\star}}(-)_\star \colon\, \lqfa \to \rqua,
      \end{array}
       \end{equation*}
that are  pairwise inverse to each other, which yield pairs of antiequivalences of categories.
\end{theorem}

 On the other hand, Drinfeld functors also provide {\em equivalences}  between quantum universal enveloping bialgebroids and quantum formal series bialgebroids whose properties are summarised in the   {\em quantum duality principle}, which we shall recall now, see \cite[\S6.1]{CheGav:DFFQG}.

 \begin{definition}
   \label{Drinfeld functor vee}
 Let  $ ( F_h, A_h, s_h, t_h, \Delta_h, \pl_h ) $  be a right quantum formal series bialgebroid, let  $  J_h := \ker \pl_h  $ and  $  I_h := \pl_h^{-1}(hA_h) = J_h + h F_h  $.  We define  $ F_h^\vee $  to be the  $h$-adic  completion of
 $$
 (F_h)^\times   :=
 F_h  +  \Sum\limits_{n \in \N_+}  h^{-n} J_h^{n}   =   F_h  +  \Sum\limits_{n \in \N_+}  h^{-n} I_h^{n}    \subseteq   A_{((h))} \otimes_{A_h} F_h.
 $$
A parallel definition applies to  \textsl{left\/}  quantum formal series bialgebroids.
\end{definition}

 The main result about the previous definition is the following:

\begin{proposition}
  \label{prop: Uvee=QUEAd}
  With the above assumptions,  $ F_h^\vee $  is a right (resp.\ left) quantum universal enveloping bialgebroid. This constructions extends to morphisms such that one obtains two different functors
  $$
  (-)^\vee \colon
  \rqfa \to \rqua
 \qquad
  \mbox{and}
  \qquad
  (-)^\vee \colon
  \lqfa \to \lqua,
  $$
  denoted the same way, and referred to as (first) Drinfeld functors.
\end{proposition}

\noindent   These, in turn, permit to introduce the second (type of) Drinfeld functor, {see \cite[\S6.2]{CheGav:DFFQG}:}

\begin{definition}
  \label{def: Uh'}
  The functors
  $$
(-)' \colon \lqua \to \lqfa \qquad \mbox{and} \qquad (-)' \colon \rqua \to \rqfa,
  $$
  given by
  $$
  (-)':= {}_* (-) \circ (-)^\vee \circ (-)^*
  = {}^* (-) \circ (-)^\vee \circ (-)_*,
  $$
  and, respectively,
  $$
  (-)':= (-)_* \circ (-)^\vee \circ {}^*(-) = (-)^* \circ (-)^\vee \circ {}_*(-),
  $$
again both denoted the same way, will be referred to as {\em (second) Drinfeld functors}.
\end{definition}

\begin{rem}
  As a byproduct of the previous definition, Drinfeld functors happen to be dual to each other, in the following sense.  If  $   F_h \in \rqfa   $  and  $   U_h \in \lqua   $  are linked by  \mbox{$   U_h = {}^\star F_h   $}  and  $   F_h = {U_h}_*   $,  then for their values under Drinfeld functors we have a similar duality relationship  $   U_h^{ \prime} = {}^*( F_h^\vee )   $  and  $   F_h^\vee = ( U_h^{ \prime} )_\star   $, see {\em loc.~cit.} for further details.
\end{rem}

To sum up, Drinfeld functors connect the categories of quantum groupoids of either type (in a covariant  way, in contrast to contravariance from linear duality). Their properties are even stronger, culminating in the  {\em quantum duality principle} for quantum groupoids, that is:

       \begin{theorem}[Quantum duality principle for quantum groupoids]
     \label{thm: QDP}
   If  $   F_h \in \rqfa   $  is a quantisation of $ J\!L   $ for a Lie-Rinehart algebra $(L,A)$, with  $ L $  finitely generated projective as a left  $ A $-module,
 then  $   F_h^\vee \in \rqfa $  is a quantisation of $ V(L^*)^\op   $,   where  $ L^* = \Hom_A(L,A)$ is the $A$-linear dual of $L$.
  On the other hand, if  $   U_h \in \rqua   $  is a quantisation of $ V\!L ^\op   $,   then  $   U_h^{ \prime} \in\rqfa $  is a quantisation of  $ J( L^*)  $.
\end{theorem}

          Analogous statements apply to Drinfeld functors for left quantum groupoids.

          \begin{rem}
            \label{rem: L to L^*}
 In  Theorem \ref{thm: QDP}  above, the dual $L^*$ is endowed with a canonical structure of a Lie-Rinehart algebra over $A$ such that $(L,L^*)$ becomes a Lie-Rinehart bialgebra; hence, it makes sense to speak of  $ V (L^*) $  and  $ J(L^*)$,  see  \cite[\S2.2]{CheGav:DFFQG}  for further details.
\end{rem}

\subsection{Action bialgebroids constructed from quantum bialgebroids} \label{subsec: top action bialg}

In this section, we will make  precise the action bialgebroid construction in the setting of quantum universal enveloping resp.\ formal series bialgebroids, and generalise some of our preceding results.
For a right quantum formal series bialgebroid $J_h L$, there exists a unique left quantum universal enveloping bialgebroid $V_h L$ such that  $   J_h L = V_h L_*   $  and $   V_hL =   {}^{\raisebox{1pt}{$\star$}} \!(J_h L )   $.
Denote by $\partial_h$ as before the counit of $J_h L$ and set  $\cI_h:=\partial_h^{-1}(hA_h)$ as well as ${\cJ}_h := \Ker(\partial_h)   $.
We will make use of the completed  tensor product  with respect to the $h$-adic topology and  denote it by  $\widehat{\otimes}_{A_h}$.
Furthermore, denote by $\{V^n L \}_{n\in\N}$ the usual filtration of $V\!L$.
We will construct   finitely generated $A_h$-submodules $V^n_h L$ of $V_h L$ such that
\begin{itemize}
  \compactlist{50}
\item[$\bullet$]  \  $ V^n_h L  /hV^n_h L \simeq V^nL   $  as left $A$-modules;
 
\item[$\bullet$]  \  $ V_h L=\varinjlim V^n_h L   $;
 
  \item[$\bullet$]  \  $ (V^i_h L)(  V^j_h L)  \subset V^{i+j}_h L   $.
\end{itemize}
 First, remark that  $J_h L =\varprojlim J^n_h L /\! \cJ_h^{n+1}   $ and
set  $   V^n_h L :=   {}^{\raisebox{1pt}{$\star$}} \! \big(J^n_hL / {\cJ}_h^{n+1} \big)   $.
One has
$$
V^n_h L     :=   \big\{  u_h \in V_h L  \mid  \langle u_h   , \psi_h \rangle = 0 \  \forall \psi_h \in \cJ^{n+1}_h  \big\}.
$$
 Then  $ V^n_h L  $  is a quantization of $V^n L$ and
 $$
 V_h L=\varinjlim V^n_h L.
 $$
Let us study now the properties of  $   J_h L / \cJ_h^{n+1}   $ and  $ V^n_h L    $:
\begin{itemize}
  \compactlist{50}
\item[$\bullet$]  \
  $\cJ_h^{n+1}   $  is a right  ${(U_h)}_*$-module, and hence it is a right $U_h$-comodule; therefore,  $   J_h L  / \cJ_h^{n+1} $
  inherits a right $U_h$-comodule structure;
\item[$\bullet$]
  \  $\Delta ( \cJ_h^{n+1}) \subseteq \bigoplus_{i+j=n+1}
  \cJ_h^i \ \widetilde{\otimes}_{A_h}  \cJ_h^j   $, and hence $\Delta$ induces a map, still denoted
  \begin{equation*}
\Delta \colon J_hL / \cJ_h^{n+1} \to \textstyle\bigoplus\limits_{i+j=n+1} J_h L / \cJ_h^i  \ \widetilde{\otimes}_{A_h}  J_h L / \cJ_h^j.
  \end{equation*}
  \noindent Taking the dual ${^\star}(-)$, we obtain a partial product map
  \begin{equation*}
    V^i_h L  \ \widehat{\otimes} \ V^j_h L
    \to
    V^{i+j}_h L.
\end{equation*}
\end{itemize}
Let now $R_h$ be an element of  $ {}_{V_hL}\mathbf{YD}^{V_hL}$;
then $R_h$ is a right $J_h L$-module. Moreover, the partial left action of $ V^i_h L$ on $R_h$ gives rise to a
partial left coaction of  $   J_h L / \cJ^i   $  on $R_h$ because any map  $   V^i_h L  \ \widehat{\otimes}_{A_h} \, R_h \to R_h   $
can be seen as a map
$$
R_h \to \Hom_{A_h} (V^i_h L   , R_h )    =    (V^i_h L )_* \,  \widehat{\otimes}_{A_h} \, R_h.
$$
On the other hand, any map  $   R_h  \ \widehat{\otimes }_{A_h} \, V_h L  \to R_h   $  can be seen as a map
$$
R_h \to
\Hom_{A_h} \big( \varinjlim V^n_h L    , R_h \big)   =   \varprojlim (V^n_h L )_*  \ \widehat{\otimes}_{A_h} \, R_h   =   \varprojlim ( J_h L / \cJ_h^{n+1} ) \ \widehat{\otimes}_{A_h} \,  R_h.
$$
Thus, any left $V_h L $-module defines a left $J_h L$-comodule.
More precisely, any partial left $V^n L$-module defines a partial  left  $   J_h L /\cJ_h^{n+1} $-comodule.

Next,
we have to check that $R_h$ is an object in $ {}^{(V_hL)_*}\mathbf{YD}_{(V_hL)_*}$:
for that we mimic the proof of  Lemma \ref{sofocle2}\,(iv)  reasoning with  $   u \in V^n_h L    $  and  $   \psi \in J_h L /\cJ_h^{n+1}   $ , and using the properties of  $   J_h / \cJ_h^{n+1}   $  and  $ V^n_h L    $.
If $R_h$ is a braided commutative monoid in  $ {}_{V_h L }\mathbf{YD}^{V_h L }   $,
then
it is also a braided commutative monoid in
$ {}^{(V_hL)_*}\mathbf{YD}_{(V_hL)_*}$, with
$   (V_h L )_* = J_h L   $.
Indeed, the proof of Lemma \ref{rosamunde} works again in our situation if we make the computation in  $ V^n_h L $,  with the latter being a finitely generated projective $A_h$-module.  Then, one can perform the action bialgebroid construction of  $   R_h  \, \widehat{\#} \, V_h L    $,  respectively of  $   J_h L  \, \widehat{\#}  \, R_h   $,  and Theorem \ref{vaexholm} applies to this topological action bialgebroid context.
We need to check that the morphism
$$
  (V_h L )_*  \, \widetilde{\#} \,  R_h    \to    \Hom _{A_h} \big (R_h \, \widehat{\#} \, V_h L    ,   R_h \big),
   \quad
   \psi \, \widetilde{\otimes}_{A_h} r    \mapsto    \big\{  r' \, \widehat{\otimes}_{A_h}  u  \mapsto  r' \cdot (\langle \psi , u \rangle  \lact  r)  \big\}
$$
of right bialgebroids over $R_h   $ is an isomorphism. For this,
as $ V^n_h L  $  is a finitely generated $A_h$-module, one easily sees that the analogous morphism is an isomorphism when replacing $V_hL$ by $ V^n_h L  $, and
we conclude by taking the projective limit.

\subsection{Drinfeld functors on action quantum groupoids}
Returning to our study of action bialgebroids, we focus now on those of quantum type, with the purpose of examining the effect of applying Drinfeld functors to this special type of quantum groupoids.  In a nutshell, the outcome will be the following:

\smallskip

\begin{center}
\noindent {\em Drinfeld functors commute with the action bialgebroid construction}.
\end{center}

\medskip

\noindent In order to achieve such a statement, let us state a technical result first:

\begin{lemma}
  \label{B[smash]1 commut 1[smash]F}
 Let  $(U,A)$  be a left bialgebroid and $ R $ a braided commutative monoid in $\ydu$. Then in $R \# U$ we have, for all  $  r \in R  $  and  $  u \in U  $,
\begin{equation*}
 (r \otimes_A 1_U)(1_R \otimes_A u) = (1_R \otimes_A u)(r \otimes_A 1_U)  \ \xLeftrightarrow{\phantom{M}} \   \ker \gve \cdot R = 0,
\end{equation*}
where here on the right hand side $\cdot$ denotes the left $U$-action on $R$.
A similar statement can be made when talking about right bialgebroids.
\end{lemma}

\begin{proof}
 Note that because of the splittings  $s(A) \oplus \ker \gve = U = t(A) \oplus \ker \gve $, every  $ u \in U $  can be uniquely written as
  $ u = t \gve(u) + u^{t} $  and  $ u = s \gve (u) + u^{s} $, where  $ u^{t} := u - t \gve(u) $  and  $ u^{s} := u - s \gve(u)$ such that $u^{s}, u^{t} \in \ker \gve$. One has:
  \begin{eqnarray*}
    (r \otimes_A 1_U)(1_R \otimes_A u) - (1_R \otimes_A u)(r \otimes_A 1_U)
    &  {\overset{\scriptscriptstyle{
\eqref{orangen1}
}}{=}}
&
    r \otimes_A u - u_{(1)} r \otimes_A u_{(2)}
    \\
  &  {\overset{\scriptscriptstyle{
\eqref{ydforget2}
}}{=}}
&
    r \gve(u_{(1)})  \otimes_A u_{(2)} - u_{(1)} r \otimes_A u_{(2)}
    \\
  &  {\overset{\scriptscriptstyle{
}}{=}}
&
  - u^{t}_{(1)} r \otimes_A u_{(2)}
    \\
  &  {\overset{\scriptscriptstyle{
}}{=}}
&
  - u^{t}_{(1)} r \otimes_A \big(u_{(2)} - s\gve(u_{(2)})\big)
- u^{t}_{(1)} r \otimes_A s\gve(u_{(2)})
  \\
  &  {\overset{\scriptscriptstyle{
}}{=}}
&
  - u^{t}_{(1)} r \otimes_A u^{s}_{(2)}
- u^{t} r \otimes_A 1_U.
  \end{eqnarray*}
  Hence, when the left hand side of this identity is zero, its right hand side is so was well, but on the right hand side the second summand belongs to $ \ker \gve \cdot R \otimes_A 1_U $, the first one to
  $ \ker \gve \cdot R \otimes_A \ker \gve$. Since
  $ U = 1_U \oplus \ker \gve $,  this implies that the single summands are already zero, and from the second one deduces, in particular, that
  $u^{t} r = 0$. Since  $ u^{t} $ is the generic element in  $ \ker \gve $ as  $ u $  ranges through all of  $ U $, the claim follows.
\end{proof}

Adapting the proof, the previous statement holds true for topological bialgebroids such as quantum universal enveloping bialgebroids:

\begin{corollary}
  \label{cor:B[smash]F h-commutative}
 Let  $ A_h $  and  $ R_h $  be two  $ h $-topologically  complete $ k[[h]]$-algebras,  let $ (F_h, A_h) $  be a (possibly topological) left bialgebroid, and $ R_h $ a braided commutative monoid in $\ydfh$.  Then in $R_h \# F_h$ we have for all $r \in R_h$ and $f \in F_h$:
\begin{equation*}
    (r \otimes_{A_h} 1_{F_h})(1_{R_h}  \otimes_{A_h} f) = (1_{R_h}  \otimes_{A_h} f)(r  \otimes_{A_h} 1_{F_h}) \!\!\! \mod \! h
    \ \xLeftrightarrow{\phantom{M}} \   \ker \gve_h \cdot R_h \in h R_h.
\end{equation*}
A similar statement can be made when talking about right bialgebroids.
\end{corollary}

The key consequence we are interested in is the following:

\begin{theorem}
  \label{thm: vee}
     Let $(F_h, A_h) \in \lqfa$ be a left quantum formal series bialgebroid and 
     $ R_h $ an  $ h $-topologically  complete $ k[[h]]$-algebra
that is a braided com\-mutative monoid 
in  $ \ydfh  $, and 
such that  $  R_h / hR_h  $  is commutative. 
  Moreover, assume that
\begin{equation*}
   \quad   fr = 0  \!\!\! \mod \! h
\end{equation*}
for all $f \in \ker \gve_h \subseteq F_h$ and $r \in R_h$.
Then $ R_h $  is a braided commutative YD algebra over  $ F_h^\vee $ as well, for which the  $ F_h^\vee $-coaction  $ \rho^\vee_{R_h} $  has the additional property
$$
\rho^\vee_{R_h}(r) = r \otimes_{A_h} 1_{U_h} \!\!\! \mod \! h
$$
for all $r \in R_h$, and there exists a canonical isomorphism
$$
R_h  \# F_h^\vee \xrightarrow{\ \simeq \ }  ( R_h  \# F_h )^{\vee}
$$
of  (topological)  left bialgebroids over  $ R_h $  that is uniquely determined by
$$
r  \otimes_{A_h} ( h^{-1} f)   \mapsto   h^{-1} ( r  \otimes_{A_h}  f)
$$
for all $  r \in R_h $ and all $ f \in \ker \gve_h  $.
In particular,  $  R_h \# F_h^\vee  $  is a left quantum universal enveloping bialgebroid over  $ R_h  $  whose semiclassical limit is  $ R \# V (L^*) $
if the semiclassical limit of  $ F_h $  is  $ J\!L $  and that of  $ R_h $  is  $ R  $;  hence, $ R_h \# F_h^\vee $  is a quantisation of  $ R \# V(L^*)  $.
\end{theorem}

\begin{proof}
With respect to the counit $\gve_h$ of $F_h$, let us again write
$ I_h = \gve_h^{ -1}( h A_h )   $,
and note as above that  $   I_h = h F_h + \ker \gve_h   $.  Recall furthermore that
$  F_h^{ \vee} = \widehat{F_h^{ \times}}$, where the latter denotes the $ h $-adic
completion of  $ F_h^{ \times} $, see
Definition \ref{Drinfeld functor vee}.
The right  $ F_h $-coaction  on  $ R_h $  clearly coextends to a right  $ F_h^{ \vee} $-coaction.  If the left  $ F_h $-action  on  $ R_h $  also extends to a left action by  $ F_h^{ \vee} $,  then these coaction and action will obviously be compatible (as the original ones were) so as to turn  $ R_h $  into a braided commutative YD algebra over  $ F_h^{ \vee} $,  as expected.  Hence, it is enough to prove that the  $ F_h $-action  on  $ R_h $  does extend to an action by  $ F_h^{ \vee} $.

Since  $ F_h^{ \vee} $  is (topologically) generated by  $ h^{-1} I_h = F_h + h^{-1} \ker \gve_h   $,
it is enough to show that we can extend the action of  $ F_h $  on  $ R_h $  to elements of  $   h^{-1} \ker \gve_h   $.  Now, this amounts to  $   ( h^{-1} \ker \gve_h) \cdot R_h  \subseteq  R_h   $, which is equivalent to  $  (\ker \gve_h ) \cdot R_h \subseteq h R_h  $,  which is true by assumption.

   Next, let us abbreviate  $  \mathcal{F}_h := R_h \# F_h  $.  As before,  $  \mathcal{F}_h^{ \vee} = \widehat{\mathcal{F}_h^{\times}}  $  is the $ h $-adic   completion of  $ \mathcal{F}_h^{ \times}  $,
where as above
$$
\mathcal{F}_h^{ \times} = s(R_h) + \Sum\limits_{n>0}   ( h^{-1} I_{\mathcal{F}_h} )^n = R_h \otimes_{A_h} F_h + \Sum\limits_{n>0}   ( h^{-1} \ker \gve_{\mathcal{F}_h} )^n
$$
with
$   I_{\mathcal{F}_h}   = \gve_{\mathcal{F}_h}^{ -1}\big( h^{-1} R_h \big) = h   \mathcal{F}_h + \ker \gve_{\mathcal{F}_h}  .
$
Note also that, by the very definition of the counit for  $   \mathcal{F}_h $,  we have  $  \ker \gve_{\mathcal{F}_h} = \ker(\id_{R_h} \otimes_{A_h} \gve_h) \supseteq R_h \otimes_{A_h}  \ker \gve_h  $.  Then observe
that for any  $   r, r' \in R_h   $  and  $   f, f' \in \ker \gve_h   $,  one has
$$
(r \# f)  ( r' \# f' )    =    r \cdot  (f_{(1)}r') \# f_{(2)}  f'     \in
\pig( h R_h \otimes_{A_h} \ker \gve_h + R_h \otimes_{A_h} (\ker \gve_h)^{ 2}  \pig)
$$
since  $  \Delta (f) \in \Delta( \ker \gve_h )   \subseteq  (\ker \gve_h \otimes_{A_h} F_h + F_h \otimes_{A_h} \ker \gve_h )   $  and  Corollary  \ref{cor:B[smash]F h-commutative}  applies.
This yields
$$
(\ker \gve_{\mathcal{F}_h})^2    =    (R_h \otimes_{A_h} \ker \gve_h) \cdot (R_h \otimes_{A_h} \ker \gve_h)    \subseteq    h   R_h \otimes_{A_h} \ker \gve_h + R_h \otimes_{A_h} (\ker \gve_h)^2,
$$
 and then by iteration
 $$
(\ker \gve_{\mathcal{F}_h})^n    \subseteq    \Sum\limits_{s=0}^{n} h^{n-s}   R_h \otimes_{A_h} (\ker \gve_h)^s   =   R_h \otimes I_{  F_h}^{ n}
 $$
for all  $   n \in \N^+   $,  while the converse inclusion is obvious.  Therefore, we obtain
$$
\mathcal{F}_h^{ \times}    =    R_h \otimes_{A_h} F_h + \Sum\limits_{n>0}
( h^{-1}  \ker \gve_{\mathcal{F}_h} )^n    =    R_h \otimes_{A_h} F_h + R_h \otimes_{A_h} {\textstyle \sum\limits_{n>0}} h^{-n} I_{  F_h}^{ n} = R_h \otimes_{A_h} F_h^{ \times},
$$
whence  $   \mathcal{F}_h^{ \vee}   =   R_h  \widehat{\otimes}_{A_h}  F_h^{ \vee}   $
as  $ k[[h]] $-modules as in the last claim, and which extends to the bialgebroid structures.
As a last step, denoting by  $ \rho_{R_h} $  the right  $ F_h $-coaction  on  $ R_h  $, one has
$$
\rho_{R_h}(r)   =   r_{[0]} \otimes_{A_h} r_{[1]}  =   r \otimes_{A_h} \! 1 + r_{[0]} \otimes_{A_h} \big( r_{[1]} - s \gve_h (r_{[1]} ) \big)
$$
with  $  r_{[0]} \otimes_{A_h} \big(r_{[1]} - s \gve_h(r_{[1]}) \big)
\in R_h \otimes_{A_h} \ker \gve_h  $.
 Then, denoting by  $ \rho^\vee_{R_h} $  the  right  $ F_h^\vee$-coaction  on  $ R_h $  induced by  $ \rho_{R_h} $,  we have
$$
  \rho^\vee_{R_h}(r)  =   r \otimes_{A_h} \! 1 + h \pig( h^{-1} \big( r_{[0]} \otimes_{A_h} \big( r_{[1]} - s\gve_h(r_{[1]}) \big) \big) \pig)
$$
  with  $ r_{[0]} \otimes_{A_h} h^{-1} \big( r_{[1]} - s\gve_h (r_{[1]}) \big)  \in R_h \otimes_{A_h}  h^{-1} \ker \gve_h \subseteq R_h \otimes_{A_h} \! h^{-1} I_{h} $,  and therefore we eventually end up with  $ \rho^\vee_{R_h}(r) = r \otimes_{A_h} \! 1  \! \mod \! h $.

   Finally, the last statement about  $ R_h \# F_h^\vee $  being a quantisation of  $ R \# V(L^*) $ follows at once from the above along with  Theorem \ref{thm: QDP}.
\end{proof}

Needless to say, a similar statement can be made when starting from right quantum bialgebroids.
   The dual result, concerning the Drinfeld functor  $ (-)'  $,  reads as follows:

   \begin{theorem}
     \label{thm: primeoprime}
Let $(U_h,A_h) \in \lqua $ be a left quantum universal enveloping bialgebroid
and 
$ R_h $  an  $ h $-topologically  complete $ k[[h]]$-algebra
that is a braided com\-mutative monoid 
in  $ \yduh  $, and 
such that  $  R_h / hR_h  $  is commutative.
Moreover, assume that
 \begin{equation*}
    \rho_{R_h}(r) = r \otimes_{A_h} \! 1_{U_h}  \!\!\! \mod \! h
 \end{equation*}
 with respect to the right $U_h$-coaction  $ \rho_{R_h} $  on $R_h$.  Then $ R_h $  is a braided commutative YD algebra over  $ U'_h $  for which the  $ U_h' $-action  on  $ R_h $  has the property
  $$  ur = 0  \!\!\! \mod \! h  $$
for all $ u \in \ker \gve_{U'_h}$ and  $r \in R_h$. Furthermore, there exists a canonical isomorphism
$$
R_h  \# U_h'  \xrightarrow{\ \simeq \ }  (R_h  \# U_h )'
$$
of left bialgebroids over  $ R_h $,  given by  $  r \otimes_{A_h} u
\mapsto r \otimes_{A_h} u  $  for all  $  r \in R_h  $ and  $  u \in U'_h  $.
In par\-tic\-u\-lar,  $  R_h \# U_h^{\prime}  $  is a left quantum formal series bialgebroid over  $ R_h  $  whose semiclassical limit is  $ R \# J(L^*) $
if the semiclassical limit of  $ U_h $  is  $ V\! L $  and that of  $ R_h $  is  $ R $, hence,  $ R_h \# U_h^{\prime} $  is a quantisation of  $ R \# J(L^*)  $.
\end{theorem}

\begin{proof}
  That $R_h$ is a braided commutative YD algebra over $U_h'$ if it so over $U_h$ follows from the very definition of  $ (-)' $  along with  Theorem \ref{thm: vee}  and (a right dual version of)  Lemma \ref{rosamunde},  after adapting the latter to the present setup of the topological bialgebroid  $ U_h $  which is not finite projective as an  $ A_h $-module but rather a direct limit of finite projective submodules, that is, ind-projective ({\em cf.}\
  \S\ref{limits}).
 The claimed isomorphism then follows via a chain of duality isomorphism. More precisely,
  \begin{equation*}
    \begin{split}
      (R_h \# U_h)^{ \prime} &\xrightarrow{\ = \ } \raisebox{-2pt}{${}_*$} \pig( \hskip -.8pt \big((R_h \# U_h)^{ * }\big)^{\!\vee} \hskip -.5pt \pig)  \xrightarrow{\ \simeq \ }
      \raisebox{-1.5pt}{${}_*$} \big( (U_h^{ * } \# R_h)^\vee \big)
      \\
      &
      \xrightarrow{\ \simeq \ }
      \raisebox{-1.5pt}{${}_*$} \big( (U_h^{ * })^\vee \# R_h \big)
      \xrightarrow{\ \simeq \ }
     R_h \# \raisebox{-1.5pt}{${}_*$} \big((U_h^{ * })^\vee\big)
\xrightarrow{\ = \ }
     R_h \# U_h',
    \end{split}
  \end{equation*}
  where the first and the last arrow are just the definitions, the second arrow is the isomorphism from Theorem \ref{vaexholm} once suitably adapted to the present case of an ind-projective bialgebroid  $ U_h $
(as explained in  \S\ref{subsec: top action bialg}),
  the third arrow is the isomorphism from Theorem \ref{thm: vee} right above, and the fourth one is the isomorphism arising from a right bialgebroid version of Theorem \ref{vaexholm} again.
More in detail, we are applying the {\em right}  (quantum) bialgebroid version of  Theorem \ref{thm: vee} to  $ F_h := U_h^*  $.  In this case, the condition
$ rf = 0  \!\! \mod \! h $  for the  (right)  action of  all $  U^*_h  $  on  $ R_h $  is fulfilled because of the following chain of identities:
$$
rf   =   r_{[0]}  {\mkern 1mu} \langle r_{[1]}  , f \rangle    \equiv_{(\text{mod} \ h)}  r  {\mkern 1mu} \langle 1_{U_h}  , f \rangle   \equiv_{(\text{mod} \ h)}   r {\mkern 1mu} \pl (f),
$$
where $\pl$ denotes the right counit of $U^*_h$, analogous to Eq.~\eqref{weiszjanich}.
 This implies  $  rf = 0 \!\! \mod \! h  $  for all  $  f \in \ker \pl  $  and $  r \in R_h  $.

 Finally, let us check that the additional property
 \begin{equation}
   \label{eq: u.r=0 mod hU'}
   ur = 0  \!\!\! \mod \! h
   \end{equation}
 is fulfilled for all $ u \in \ker \gve^{\raisebox{-1pt}{$\scriptstyle\prime$}}$ and $r \in R_h$, where $  \gve^{\raisebox{-1pt}{$\scriptstyle\prime$}} := \gve_{U'_h}$. Abbreviating $\gve := \gve_{U_h}$,
recall
that the left bialgebroid  $ U_h $  splits, as a left  $ A_h $-module,  into  $  \due {U_h} \lact {} = s(A_h) \oplus \ker \gve  $.  Similarly, the right bialgebroid  $  U^*_h $  splits, as a right  $ A_h $-module,  into  $   {U^*_h}_\ract = t^r(A_h) \oplus \ker \pl  $.
With respect to the natural pairing
  $ \langle \cdot , \cdot \rangle \colon U^*_h \times U_h \to A_h  $ as described in \eqref{duedelue},
the very definition  $  U'_h := {}_*((U^*_h)^{{\mkern -1mu}\vee})  $  implies
  $  \langle (U^*_h)^\vee, U'_h   \rangle \subseteq A_h  $, and therefore
  $  \langle h^{-1} \ker \pl, U'_h  \rangle \subseteq A_h  $, which, in turn,
means
  $  \langle  \ker \pl, U'_h   \rangle \subseteq hA_h  $.
 Now look at  $  \ker \gve^{\raisebox{-1pt}{$\scriptstyle\prime$}} = U'_h \cap \ker \gve  $.  On the one hand, we have
  $  \langle  \ker \pl, \ker \gve^{\raisebox{-1pt}{$\scriptstyle\prime$}}   \rangle \subseteq hA_h  $
 by the previous analysis; on the other hand, the definitions give
  $  \langle  1_{U^*_h}, \ker \gve  \rangle = 0  $,
 hence
  $  \langle  1_{U^*_h}, \ker \gve^{\raisebox{-1pt}{$\scriptstyle\prime$}}   \rangle = 0  $ as well.
 This together with the splitting  of $  {U^*_h}_\ract $ as above leads to the conclusion that
  $  \langle   U^*_h, \ker \gve^{\raisebox{-1pt}{$\scriptstyle\prime$}}  \rangle \subseteq hA_h  $,
 hence  $  \ker \gve^{\raisebox{-1pt}{$\scriptstyle\prime$}} \subseteq h \mkern 1mu ({}_\star(U_h^* )) = hU_h  $.  The outcome is  $  \ker \gve^{\raisebox{-1pt}{$\scriptstyle\prime$}} \subseteq hU_h  $,  which obviously implies  \eqref{eq: u.r=0 mod hU'}.

   At last, the final part of the statement, that is, the one about  $ R_h \# U_h^{\prime} $  being a quantisation of  $ R \# J(L^*)  $, follows from the above together with  Theorem \ref{thm: QDP}.
\end{proof}

Again, starting from right quantum universal enveloping bialgebroids yields an analogous result.
To conclude, let us sum up the content of  Theorems \ref{thm: vee} \& \ref{thm: primeoprime} in a unified and concise statement;  to this end, let us introduce the following subcategories first:

\begin{definition}
  \label{def: action-quantum-groupoids}
 Let  $ R_h $  be any  $ h $-topologically  complete $ k[[h]]$-algebra  such that  $  R_h / hR_h  $  is commutative.  We denote by:
 \begin{enumerate}
   \compactlist{50}
 \item
   $A_h\mbox{-}\lqfr$ the subcategory of  $  \lqfr $  whose objects are quantum groupoids of the form  $  R_h \# F_h  $  for some  $  F_h \in \lqfa $ such that $R_h$ is a braided commutative monoid in $\ydfh$;
\item
  $A_h\mbox{-}\lqur$ the subcategory
of  $  \lqur $  whose objects are quantum groupoids of the form  $  R_h \# U_h  $  for some  $  U_h \in \lqua $ such that $R_h$ is a braided commutative monoid in $\yduh$.
  \end{enumerate}
 Objects in these categories  will be referred to as {\em (left) action quantum groupoids}.
\end{definition}

As always, a similar definition applies when replacing {\em left} by {\em right} everywhere.
  Our recapitulatory statement then reads as follows:

  \begin{theorem}
    \label{thm: recap-Drinf-funct_act-q-grpd's}
    Let  $ R_h $  be any  $ h $-topologically  complete $ k[[h]]$-algebra  such that  $  R_h / hR_h  $  is commutative.
    Then the two Drinfeld functors  $ (-)^\vee \colon
  \lqfr \to \lqur
    $
    and
    $  (-)' \colon
   \lqur \to \lqfr
    $
restrict to functors
$$
(-)^\vee \colon
A_h\mbox{-}\lqfr \to A_h\mbox{-}\lqur   ,
\qquad
(-)' \colon
A_h\mbox{-}\lqur \to A_h\mbox{-}\lqfr   ,
$$
 which are isomorphisms (between categories) that are inverse to each other. A parallel, right-handed statement holds true as well.
\end{theorem}

  \begin{rem}
    \label{rem:smash quantum = quantum smash}
    Directly by construction and/or from our results above, one sees that  $  R_h \# H_h  $,  resp.\  $  H_h \# R_h  $,  is a left, resp.\ right, quantum groupoid over  $ R_h  $,  that is, a quantum universal enveloping or formal series bialgebroid if $ H_h  $ is so.  More precisely, every action quantum groupoid is a ``quantum action groupoid'', in the following sense: with assumptions as in  Definition \ref{def: action-quantum-groupoids}  above,
 if  $ H_h $ is a quantisation of the (classical) bialgebroid  $ H $  over  $  A := A_h / h A_h  $,   then  $  R := R_h / h R_h  $  is a braided commutative YD algebra
 for  $ H  $,  and  $  R_h \# H_h  $,  resp.\  $  H_h \# R_h  $,  is a quantisation of the (classical) action  bialgebroid  $  R \# H  $,  resp.\  $  H \# R  $.
\end{rem}

\begin{exs}  \label{pomeriggio}
\
 \begin{enumerate}
    \compactlist{99}
  \item
    Let  $ F_h := F_h[[G]] $  be a quantum formal series Hopf algebra over a Poisson (algebraic~or Lie) group  $ G $; hence, it is a particular type of (left and right) quantum formal series bialgebroid, {\em cf.}\  \cite{CheGav:DFFQG}  and references therein.  As every Hopf algebra is a braided commutative YD algebra over itself (either by means of the adjoint action and the coproduct or by the product and the coadjoint coaction), one obtains an action bialgebroid  $(F_h \# F_h, F_h)$  whose semiclassical limit is  $  F[[G]] \# F[[G]] = F[[G \ltimes G]]  $:  therefore, we can see  $F_h \# F_h$  as a quantisation
    of the action groupoid  $ G \ltimes G $;
    this action bialgebroid  $(F_h \# F_h, F_h)$ whose underlying algebra structure is that of the {\em Heisenberg double}, could be termed  {\em Weyl Hopf algebroid}, see Example \ref{sable}, while the underlying action groupoid  $ G \ltimes G $  is the
      {\em symplectic double}  of  $ G $  (see, for instance, \cite{Lu:MMATQL, Lu:HAAQG}).

   Now, let us apply  Theorem \ref{thm: vee}  to the action quantum groupoid  $  F_h \# F_h  $: the outcome is that  $( F_h \# F_h)^\vee = F_h \# F_h^{\vee} = F_h \# U_h  $, where  $  U_h = U_h(\g^*) := F_h^{\vee}  $, is a new action quantum groupoid (actually, a left quantum universal enveloping bialgebroid) whose semiclassical limit is a (left) action bialgebroid of the form  $  F[[G]] \# U(\g^*) = V( G \rtimes \g^*)  $,  with respect to a well-defined structure of
   braided commutative YD algebra
 over  $ U(\g^*) $ on  $ F[[G]] $.  In short,  $ ( F_h \# F_h )^\vee  = F_h \# F_h^{\vee}  $  is a quantisation of the action Lie-Rinehart algebra  $  F[[G]] \rtimes \g^*  $,  or also (in an infinitesimal sense) of the action groupoid  $  G \rtimes G^*  $.
  \item
    Let  $  U_h := U_h(\g)  $  be a quantum universal enveloping algebra over the Lie bialgebra  $ \g  $,  so it is a particular type of (left and right) quantum universal enveloping bialgebroid,
    {\em cf.}\ again \cite{CheGav:DFFQG}.
    Note that  $ U_h $  is a direct limit of a directed system of finite projective (left and right)  $ A_h $-modules:  then the  definition/construction of left and right dual bialgebroids can be repeated for  $ U_h $  (as explained in  \S\ref{limits},  also taking into account that we are now in the simpler context of (topological) Hopf algebras rather than bialgebroids) so that the dual  $  F_h = F_h[[G]] := \big( U_h(\g) \big)^* = \big( U_h(\g) \big)_*  $  is a quantum formal series Hopf algebra over the (formal) group  $ G $  associated to  $ \g  $,  in the sense of  \cite{CheGav:DFFQG}; note that,  in this case, right and left duals coincide.
    In addition (again because  $ U_h $  is a direct limit of finite projective  $ A_h $-modules),  we have an identification  $  U_h = F_h^{\star}  $  (a suitable topological dual) and, in addition, also the proof of  Lemma \ref{rosamunde} adapts to the present setup.  Therefore, by part (i),  we first consider  $ F_h $  as a braided commutative YD algebra over itself and consider the left (action) bialgebroid  $  (F_h \# F_h, F_h)  $;  second, applying  Lemma \ref{rosamunde},  we find that  $ F_h $  is also a braided commutative YD algebra over  $ U_h  $.  This, in turn, leads us to introduce the right (action) bialgebroid  $ (U_h \# F_h, F_h)  $  whose semiclassical limit is  $  U(\g) \# F[[G]] = V\big(\g \ltimes F[[G]]\big)  $:  therefore, we can see  $ U_h \# F_h $  as a quantisation of the symplectic groupoid  $  G \ltimes G  $,  or rather of the associated (tangent) action right Lie-Rinehart algebra  $  \g \ltimes F[[G]]  $  over the base algebra  $ F[[G]]  $.

   Now, let us apply (the right-handed version of) Theorem \ref{thm: primeoprime}  to the action quantum groupoid  $  U_h \# F_h  $.  Writing  $  F_h[[G^*]] := U_h^{\prime}  $,  the outcome now is that  $  ( U_h \# F_h)'  = U_h^{\prime} \# F_h = F_h[[G^*]] \# F_h  $  is a new action quantum groupoid (actually, a right quantum formal series bialgebroid) whose semiclassical limit is a (right) action bialgebroid of the form  $  F[[G^*]] \# F[[G]] = F[[ G \rtimes G^*]]  $,  with respect to a well-defined structure of
   braided commutative YD algebra
 over  $ F[[G^*]] $  on  $ F[[G]] $.  In short,  $  ( U_h \# F_h )'  = U_h^{\prime} \# F_h  $  is a quantisation of the action groupoid  $  G \rtimes G^*  $.
  \item
    Comparing the constructions in  parts (i) \& (ii) above, if we take  $ F_h $  and  $ U_h $  dual to each other, that is, $  F_h = U_h^{\ast}  $  and  $  U_h = F_h^{\star}  $,  then the two action bialgebroids $  F_h \# F_h  $  and  $  U_h \# F_h  $  are connected via  Lemma \ref{rosamunde}:  in particular, they are dual to each other again.  But then,  by  Lemma \ref{rosamunde},  Theorem \ref{thm: vee},  and  Theorem \ref{thm: primeoprime}, we find that also the action bialgebroids  $  (F_h \# F_h)^\vee = F_h \# F_h^\vee  $  and  $  (U_h \# F_h)' = U_h^\prime \# F_h  $  are dual to each other.  Therefore, the same happens with their semiclassical limits  $  F[[G]] \# F[[G^*]]  $  and  $
    U( \g^*) \# F[[G]]  $.
    The outcome is that the mutually dual quantum action groupoids  $  F_h \# F_h^\vee  $  and  $  U_h^\prime \# F_h  $ provide mutually dual quantisations of a common, underlying geometrical datum,
 the action groupoid  $  G \rtimes G^*  $.
  \item
From a different point of view, it is explained in  \cite{Lu:MMATQL, Lu:HAAQG} that all quantum action bialgebroids presented in  (i), (ii),   and  (iii)  above can be seen as providing a quantisation of the {\em dressing action}  of  $ G^* $  on  $ G $  and of  $ G $  on  $ G^* $.
 \end{enumerate}
\end{exs}

\appendix

\section{Comodules and YD modules over bialgebroids}
\label{A}

The next few sections lists a multitude of technical details needed in the main text. With the exception of \S\ref{naemlichhier}, all of this is standard material (or can be more or less directly deduced from such), and can be found anywhere in the bialgebroid literature.

As for the categories $\umod$ and $\modu$ of left resp.\ right modules over a left bialgebroid $(U, A, s, t, \gD, \gve)$, although the first one is monoidal and the second one is not, in both cases one has a forgetful functor $\umod \to \amoda$ resp.\ $\modu \to \amoda$ with respect to which we denote
\begin{equation}
  \label{soedermalm}
a \lact m \ract b := s(a)t(b)m, \qquad b \blact n \bract a := ns(a)t(b)
  \end{equation}
for $a,b\in A$, $m \in \umod$, and $n \in \modu$.

\subsection{}
\label{schokoladenraspel1}
\hspace*{-.3cm}
A {\em right comodule} over a left bialgebroid $(U,A)$ is a pair $(M,\rho_M)$ of a right $A$-module
and a
right $A$-linear $U$-coaction $\rho_M\colon M\to M \otimes_A \due U \lact {}$, which induces a left $A$-action with respect to which the coaction becomes linear as well and corestricts to the Takeuchi subspace $M \times_A U$:
\begin{equation}
  \label{ha3}
\rho_M(a m a' )
= m_{[0]} \otimes_A a \blact m_{[1]} \ract a',
\quad
am_{[0]}\otimes_A m_{[1]}  = m_{[0]}\otimes_A m_{[1]} \bract a,
\end{equation}
where $am := m_{[0]} \gve(m_{[1]}\bract a)$ is the induced left $A$-action and $\rho_M(m) =: m_{[0]}\otimes_A m_{[1]}$, using a Sweedler square bracket subscript notation. The respective category is denoted by $\comodu$, while the category $\ucomod$ of {\em left} $U$-comodules is defined along analogous lines.

\subsection{}
\label{ydydydyd}
\hspace*{-.3cm}
A {\em left-left Yetter Drinfeld (YD)} module $M$ over a left bialgebroid $(U,A)$ is at the same time
a left $U$-module $U \otimes M \to M, \ (u,m) \mapsto um$ and a left $U$-comodule $\gl_M \colon M \to U_\ract \otimes_A M, \ m \mapsto m_{(-1)} \otimes_A m_{(0)}$ such that the forgetful functors $\umod \to \amoda$ and $\ucomod \to \amoda$ coincide on $M$, that is,
\begin{equation*}
 a \lact m \ract a' = ama'
\end{equation*}
for $a, a' \in A$ and $m \in M$,
and, in particular, such that for all $u \in U$ and $m \in M$
\begin{equation}
  \label{yd}
u_{(1)} m_{(-1)} \otimes_A u_{(2)} m_{(0)} = (u_{(1)} m)_{(-1)}  u_{(2)} \otimes_A  (u_{(1)} m)_{(0)}
\end{equation}
holds, 
the corresponding category of which is braided  monoidal (see below) with respect to $\otimes_A$, unit object given by the base algebra $A$, and denoted by $\yd$. There is a monoidal equivalence between this category and the (left weak) monoidal centre of the category $\umod$.

\subsection{}
\label{ydydydydr}
\hspace*{-.3cm}
Likewise,
a {\em left-right Yetter Drinfeld} module $M$ over a left bialgebroid $(U,A)$ is
at the same time a left $U$-module
$U \otimes M \to M, \ (u,m) \mapsto um$
and a right $U$-comodule $\gr_M \colon M \to M \otimes_A \due U \lact {}, \ m \mapsto m_{[0]} \otimes_A m_{[1]}$ such that the forgetful functors $\umod \to \amoda$ and $\comodu \to \amoda$ coincide on it, {\em i.e.},
\begin{equation}
  \label{ydforget2}
a \lact m \ract a' = ama'
\end{equation}
for $a, a' \in A$ and $m \in M$,
and such that for all $u \in U$ and $m \in M$
\begin{equation}
  \label{yd2}
u_{(1)} m_{[0]} \otimes_A u_{(2)} m_{[1]} = (u_{(2)} m)_{[0]}  \otimes_A  (u_{(2)} m)_{[1]}u_{(1)}
\end{equation}
holds , the corresponding category of which is braided  monoidal (see below) with respect to $\otimes_A$, unit object given by the base algebra $A$, and denoted by $\ydu$.
This time, there is a monoidal equivalence between this category and the (right weak) monoidal centre of the category $\umod$.

\subsection{}
\label{trecce1}
\hspace*{-.3cm}
A {\em monoid in $\ydu$} for a left bialgebroid $(U,A)$ is a left-right YD module $R \in \ydu$, which is also
       a monoid in $\umod$, meaning not only
\begin{equation}
  \label{marmor1}
u(r \cdot r') = (u_{(1)} r) \cdot (u_{(2)} r'), \qquad u1_R = \gve(u) \lact 1_R,
\end{equation}
but also
\begin{equation}
  \label{marmor2}
a (r \cdot r') = (a  r) \cdot r', \qquad (r \cdot r') a = r \cdot (r' a), \qquad (r a) \cdot r' =  r \cdot (a r')
%
\end{equation}
for the left and right $A$-actions on $R$.
Moreover, it is a monoid in $\comoduop$, that is,
\begin{equation}
  \label{marmor3}
(r \cdot r')_{[0]} \otimes_A (r \cdot r')_{[1]} = r_{[0]}  \cdot r'_{[0]} \otimes_A r'_{[1]}  r_{[1]}, \qquad 1_{[0]} \otimes_A 1_{[1]} = 1_R \otimes_A 1_U.
\end{equation}
      Finally, a monoid in $\ydu$ is called {\em braided commutative} if
\begin{equation}
  \label{marmor4}
r \cdot r' = r'_{[0]} \cdot (r'_{[1]} r)
\end{equation}
for $r, r' \in R$  is fulfilled.

\subsection{}
\label{schokoladenraspel2}
\hspace*{-.3cm}
Given a right $B$-bialgebroid $(V,B, s, t, \Delta,\pl)$, recall that a {\em left $V$-comodule} is a pair $(M,\lambda_M)$ of a left $B$-module
and a
left $B$-linear left $V$-coaction $\lambda_M\colon M\to V_\bract \otimes_B M$ inducing a right $B$-action such that the coaction becomes linear as well and corestricts to $V_\bract \times_B M$:
\begin{equation}
  \label{ha2}
\lambda_M(b m b' )
=  b \blact m^{(-1)} \ract b' \otimes_B m^{(0)},\quad
m^{(-1)}\otimes_B m^{(0)} b = b \lact m^{(-1)}\otimes_B m^{(0)},
\end{equation}
where $m b:=\pl(b \blact m^{(-1)}) m^{(0)}$ is the right $B$-action, and where we employed the Sweedler superscript notation $\lambda_M(m)=m^{(-1)}\otimes_B m^{(0)}$. Denote by $\vcomod$
the category of left $V$-comodules, while
$\comodv$ denotes the category of {\em right} $V$-comodules.

\subsection{}
\label{rydydydydr}
\label{rydydydydl}
\hspace*{-.3cm}
In a similar spirit as for left bialgebroids, one can define right-right and right-left YD modules for {\em right} bialgebroids, see below, but not left-right or left-left ones. For example,
a {\em right-left Yetter Drinfeld module} $M$ over a right bialgebroid $(V,B)$ is
at the same time a right $V$-module $M \otimes V \to M, \ (m,v) \mapsto mv$ and a left $V$-comodule $\gl_M \colon M \to V_\bract \otimes_B M$, $m \mapsto m^{(-1)} \otimes_B m^{(0)}$
such that the forgetful functors $\umod \to \bmodb$ and $\comodu \to \bmodb$ coincide on $M$, that is,
\begin{equation*}
b \blact m \bract b' = bmb'
\end{equation*}
for $b, b' \in B$ and $m \in M$, plus the compatibility
\begin{equation}
  \label{borghese}
m^{(-1)} v^{(1)} \otimes_B m^{(0)}   v^{(2)}
= v^{(2)} (m v^{(1)})^{(-1)} \otimes_B (m v^{(1)})^{(0)}
\end{equation}
between action and coaction. Similar comments to what was said in \S\ref{ydydydyd} \& \S\ref{ydydydydr} apply, and a similar definition can be made for {\em right-right Yetter-Drinfeld} modules over right bialgebroids.

\subsection{}
\label{trecce2}
\hspace*{-.3cm}
In an analogous spirit,
a {\em monoid in $\vyd$} for a right bialgebroid $(V,B)$ is, first of all, a right-left YD module $R \in \vyd$ which is, second,
       a monoid in $\modv$, that is,
       \begin{equation*}
(r \cdot r')   v = (r   v^{(1)}) \cdot (r'   v^{(2)}), \qquad 1_Rv = 1_R \bract \pl(v),
       \end{equation*}
       which implies
       \begin{equation*}
b (r \cdot r') b' = (b r) \cdot (r' b'),
\qquad
(r b) \cdot r' = r \cdot (b r')
        \end{equation*}
for the left and right $B$-actions on $R$ and $b, b' \in B$.
       Third, it is also a monoid in $\vopcomod$: 
       \begin{equation}
         \label{zloty3}
(r \cdot r')^{(-1)} \otimes_B (r \cdot r')^{(0)}
=
\tilde r^{(-1)}  r^{(-1)} \otimes_B r^{(0)} \cdot \tilde r^{(0)}, \qquad 1^{(-1)} \otimes_B 1^{(0)} = 1_V \otimes_B 1_R.
       \end{equation}
       A monoid in $\vyd$ is called {\em braided commutative} if
\begin{equation}
  \label{kindertotenlieder}
r \cdot r' = (r'   r^{(-1)}) \cdot r^{(0)},
\end{equation}
for $r, r' \in R$       is fulfilled.

\section{Left and right and right and left Hopf algebroids}
\label{B}

\subsection{Left and right Hopf algebroids over left bialgebroids}
\label{sokc}

For a left bialgebroid $(U, A)$, consider the maps
\begin{equation*}
\begin{array}{rclrcl}
\ga_\ell  \colon  \due U \blact {} \otimes_{\Aop} U_\ract &\to& U_\ract  \otimes_A  \due U \lact,
 & u \otimes_\Aop u'  &\mapsto&  u_{(1)} \otimes_A u_{(2)}  u',
 \\
\ga_r  \colon  U_{\!\bract}  \otimes_A \! \due U \lact {}  &\to& U_{\!\ract}  \otimes_A  \due U \lact,
&  u \otimes_A u'  &\mapsto&  u_{(1)}  u' \otimes_A u_{(2)},
\end{array}
\end{equation*}
of left $U$-modules.
Then $(U,A)$ is called
(a)
{\em left Hopf (algebroid)} if $ \alpha_\ell $ is invertible and (a)
{\em right Hopf (algebroid)} 
if this is the case for $\ga_r$. Abbreviating in a Sweedler spirit
\begin{equation*}
  \begin{array}{rcl}
 u_+ \otimes_\Aop u_-  & \coloneqq &  \alpha_\ell^{-1}(u \otimes_A 1),
 \\
   u_{\smap } \otimes_A u_{\smam  }  & \coloneqq &  \alpha_r^{-1}(1 \otimes_A u),
\end{array}
  \end{equation*}
with implicit summation, one
has
\begin{eqnarray}
\label{Sch1}
u_+ \otimes_\Aop  u_- & \in
& U \times_\Aop U,  \\
\label{Sch2}
u_{+(1)} \otimes_A u_{+(2)} u_- &=& u \otimes_A 1 \quad \in U_{\!\ract} \! \otimes_A \! {}_\lact U,  \\
\label{Sch3}
u_{(1)+} \otimes_\Aop u_{(1)-} u_{(2)}  &=& u \otimes_\Aop  1 \quad \in  {}_\blact U \! \otimes_\Aop \! U_\ract,  \\
\label{Sch4}
u_{+(1)} \otimes_A u_{+(2)} \otimes_\Aop  u_{-} &=& u_{(1)} \otimes_A u_{(2)+} \otimes_\Aop u_{(2)-},  \\
\label{Sch5}
u_+ \otimes_\Aop  u_{-(1)} \otimes_A u_{-(2)} &=&
u_{++} \otimes_\Aop u_- \otimes_A u_{+-},  \\
\label{Sch6}
(uv)_+ \otimes_\Aop  (uv)_- &=& u_+v_+ \otimes_\Aop v_-u_-,
\\
\label{Sch7}
u_+u_- &=& s (\varepsilon (u)),  \\
\label{Sch8}
\varepsilon(u_-) \blact u_+  &=& u,  \\
\label{Sch9}
(s (a) t (a'))_+ \otimes_\Aop  (s (a) t (a') )_-
&=& s (a) \otimes_\Aop s (a')
\end{eqnarray}
for the left Hopf structure \cite{Schau:DADOQGHA},
where in  \eqref{Sch1} the Takeuchi-Sweedler product
\begin{equation*}
\label{petrarca}
   U \! \times_\Aop \! U    \coloneqq
   \big\{ {\textstyle \sum_i} u_i \otimes u'_i \in {}_\blact U  \otimes_\Aop  U_{\!\ract} \mid {\textstyle \sum_i} u_i \ract a \otimes u'_i = {\textstyle \sum_i} u_i \otimes a \blact u'_i, \ \forall a \in A \big\},
\end{equation*}
is meant. If the left bialgebroid $(U,A)$ is right Hopf,
one analogously verifies
\begin{eqnarray}
\label{Tch1}
u_{\smap } \otimes_A  u_{\smam  } & \in
& U \times_A U,  \\
\label{Tch2}
u_{\smap (1)} u_{\smam  } \otimes_A u_{\smap (2)}  &=& 1 \otimes_A u \quad \in U_{\!\ract} \! \otimes_A \! {}_\lact U,
\\
\label{Tch3}
u_{(2)\smap } \otimes_A u_{(2)\smam  }u_{(1)}  &=& u \otimes_A 1 \quad \in U_{\!\bract} \!
\otimes_A \! \due U \lact {},  \\
\label{Tch4}
u_{\smap (1)} \otimes_A u_{\smam  } \otimes_A u_{\smap (2)} &=& u_{(1)\smap } \otimes_A
u_{(1)\smam  } \otimes_A  u_{(2)},  \\
\label{Tch5}
u_{\smap \smap } \otimes_A  u_{\smap \smam  } \otimes_A u_{\smam  } &=&
u_{\smap } \otimes_A u_{\smam  (1)} \otimes_A u_{\smam  (2)},  \\
\label{Tch6}
(uv)_{\smap } \otimes_A (uv)_{\smam  } &=& u_{\smap }v_{\smap }
\otimes_A v_{\smam  }u_{\smam  },  \\
\label{Tch7}
u_{\smap }u_{\smam  } &=& t (\varepsilon (u)),  \\
\label{Tch8}
u_{\smap } \bract \varepsilon(u_{\smam  })  &=&  u,  \\
\label{Tch9}
(s (a) t (a'))_{\smap } \otimes_A (s (a) t (a') )_{\smam  }
&=& t(a') \otimes_A t(a),
\end{eqnarray}
see \cite[Prop.~4.2]{BoeSzl:HAWBAAIAD},
where in  \eqref{Tch1} we denoted
\begin{equation*}  \label{petrarca2}
   U \times_A U    \coloneqq
   \big\{ {\textstyle \sum_i} u_i \otimes  v_i \in U_{\!\bract}  \otimes_A \!  \due U \lact {} \mid {\textstyle \sum_i} a \lact u_i \otimes v_i = {\textstyle \sum_i} u_i \otimes v_i \bract a,  \ \forall a \in A  \big\}.
\end{equation*}
For a left bialgebroid $(U,A)$ that is both left and right Hopf,
one even has the mixed relations
\begin{eqnarray}
\label{mampf1}
u_{+\smap } \otimes_\Aop u_{-} \otimes_A u_{+\smam  } &=& u_{\smap +} \otimes_\Aop u_{\smap -} \otimes_A u_{\smam  }, \\
\label{mampf2}
u_+ \otimes_\Aop u_{-\smap } \otimes_A u_{-\smam  } &=& u_{(1)+} \otimes_\Aop u_{(1)-} \otimes_A u_{(2)}, \\
\label{mampf3}
u_{\smap} \otimes_A u_{{\smam}+} \otimes_\Aop u_{\smam  -} &=& u_{(2)\smap } \otimes_A u_{(2)\smam  } \otimes_\Aop u_{(1)},
\end{eqnarray}
as proven in \cite[Lem.~2.3.4]{CheGavKow:DFOLHA}.

\subsection{Left Hopf algebroids over right bialgebroids}

Of course, the notion
of (left or right)
Hopf algebroid also exist if the underlying bialgebroid is a right instead of a left one. Since we are going to deal with the dual of left bialgebroids (which are right bialgebroids), we will need (at least one of) these concepts as well:
we say that the right bialgebroid $(V,B)$ is a
{\em left  Hopf algebroid} if the Hopf-Galois map
\begin{equation}
\label{ibianchi2}
\gb_\ell \colon V_\bract {} \otimes_B \due V \lact {} \to   V_\bract  \otimes_B  \due U \blact {}, \quad
v'  \otimes_B v  \mapsto     v^{(1)} \otimes_B v' v^{(2)},
  \end{equation}
is bijective. Again, we define the corresponding translation map by
\begin{equation}
\label{ibianchi3}
v^{\smam} \otimes_B v^{\smap}  :=  \gb^{-1}_\ell(v \otimes_B 1).
  \end{equation}
As a side comment, in case $B=k$ is central in $V$, the map $\gb_\ell$ is invertible if and only if $V$ is a Hopf algebra, and one has
$ v^{\smam} \otimes v^{\smap} = S^{-1}(v_{(2)}) \otimes v_{(1)}$.

\begin{rem}
\label{evenmoreconfusion}
This latter definition might now lead to slight confusion in terminology as saying ``left/right Hopf algebroid'' does not specify whether the underlying bialgebroid is left or right, whereas ``left/right Hopf algebroid over a left/right bialgebroid'' or the like unfortunately appears somewhat clumsy.
We still want to stress that interchanging left and right here is more than a pure exercise in chirality yoga: this determines the monoidality of the respective category of modules.
\end{rem}

Let $(V,B)$ be a left Hopf algebroid over an underlying right bialgebroid.
The following identities are true with respect to $\gb^{-1}_\ell(v \otimes_B 1)$ from \eqref{ibianchi3}:
\begin{eqnarray}
\label{Uch1}
v^{\smam} \otimes_B  v^{\smap} & \in
& V \times_B V,
\\
\label{Uch2}
v^{\smap(1)} \otimes_B v^\smam v^{\smap(2)}  &=& v \otimes_B 1 \quad \in V_{\bract}  \otimes_B  {}_\blact V,  \\
\label{Uch3}
v^{(2)}v^{(1)\smam} \otimes_B v^{(1)\smap}  &=& 1 \otimes_B v \quad \in V_\bract
\otimes_B  \due V  \lact {},  \\
\label{Uch4}
v^{(1)} \otimes_B v^{(2)\smam} \otimes_B v^{(2)\smap}  &=&
v^{\smap(1)} \otimes_B v^{\smam} \otimes_B v^{\smap(2)},
\\
\label{Uch5}
v^{\smam(1)} \otimes_B v^{\smam(2)} \otimes_B v^{\smap} &=&
v^{\smam} \otimes_B v^{\smap\smam} \otimes_B v^{\smap\smap},
\\
\label{Uch6}
(vw)^\smam \otimes_B (vw)^\smap &=& w^\smam v^\smam \otimes_\Bop v^\smap w^\smap,  \\
\label{Uch7}
v^\smam v^\smap &=& t^r (\pl (v)),
\\
\label{Uch8}
v^\smap \bract \pl(v^\smam)   &=&  v,
\\
\label{Uch9}
(s^r (b) t^r (b'))^\smam \otimes_B (s^r (b) t^r (b') )^\smap
&=& t^r(b) \otimes_B t^r(b'),
\end{eqnarray}
where in  \eqref{Uch1}  we mean the Sweedler-Takeuchi product
\begin{equation*}
   V \times_B V   :=
   \big\{ {\textstyle \sum_i} v_i \otimes  w_i \in V \otimes_B  V  \mid {\textstyle \sum_i} b \lact v_i  \otimes w_i = {\textstyle \sum_i} v_i \otimes  w_i \bract b,  \ \forall b \in B  \big\}.
\end{equation*}

\section{Duals of bialgebroids}
\label{C}

\subsection{Left and right duals}
\label{duals}
In this section, we mainly recall how the left dual
$$
{U_*} := \Hom_A(\due U \lact {}, \due A A {})
$$
for a left bialgebroid $(U,A)$ becomes a right bialgebroid if the left $A$-module $\due U \lact {}$  is finitely generated projective over $A$.
In longer expressions, we will frequently write $\langle \psi, u \rangle$ to mean $\psi(u)$ for $\psi \in {U_*}$ and $u \in U$, as this increases readability (at least in our opinion).
A discussion similar to what follows also holds for the {\em right} dual
$$
U^* := \Hom_\Aop(U_\ract, A_A),
$$
the details of which are found, {\em e.g.}, in \cite[\S2.4]{Kow:WEIABVA}. Starting from a right bialgebroid $(V,B)$ leads in an analogous way to two duals denoted $\due V * {} $ and $^*V$; see, for example, \cite[\S3.4]{CheGav:DFFQG}.

More in detail, dualising a left bialgebroid
$(U, A, s^\ell, t^\ell, \gD_\ell, \gve)$  gives a right
bialgebroid $({U_*}, A, s^r, t^r, \gD_r, \pl)$ over the same base algebra
if $\due U \lact {}$ is  assumed
to be a finitely generated left $A$-module via the source map.
To begin with, the product is given by
\begin{equation}
\label{LDMon}
(\psi \psi')(u)
:= \langle \psi',  u_{(1)} \ract \langle \psi, u_{(2)} \rangle  \rangle,
\end{equation}
where $u \in U$, $\psi, \psi' \in {{U_*}}$, which 
induces
an $A^e$-ring structure by defining source and target:
\begin{equation*}
\label{LDSource}
 s^r \colon A \to  {U_*}, \ a \mapsto \gve((\cdot) \ract a) = \gve(\cdot)a, \qquad  t^r \colon A \to  {U_*}, \ a \mapsto \gve(a \blact (\cdot)).
\end{equation*}
Denoting as in \eqref{soedermalm} for every (left or right) bialgebroid
the four $A$-actions on ${U_*}$ by
$$
a \lact \psi  \ract b = s^r(a) t^r(b) \psi , \quad a \blact \psi  \bract b =
\psi    t^r(a) s^r(b)
$$
for $\psi  \in {U_*}$, $a,b \in A$,
one obtains the identities
\begin{equation}
\label{duedelue}
\!\!\!\!\!\!\begin{array}{rclrclrcl}
\langle\psi  , a \lact u \rangle  \!\!\!&=\!\!\!&
a \langle \psi  , u \rangle, &
\langle \psi  , u \ract a \rangle  \!\!\!&=\!\!\!&
\langle a \lact \psi , u \rangle, &
\langle \psi  , a \blact u  \rangle  \!\!\!&=\!\!\!&
\langle \psi \ract a , u \rangle,
\\
\langle \psi  , u \bract a \rangle   \!\!\!&=\!\!\!&
 \langle a \blact \psi, u \rangle, &
 \langle \psi \bract a, u \rangle  \!\!\!&=\!\!\!&
\langle \psi  , u \rangle a,
\end{array}
\end{equation}
hence a {\em left $\Ae$-pairing}. The $A$-coring structure on $({U_*}, A)$, in turn, is given by
\begin{equation}
\label{weiszjanich}
\!\!\! \begin{array}{rll}
\gD_r
\!\!\! &
\colon {U_*} \to \Hom_A(U_\bract \otimes_A \due U \lact {}, A), \
&\!\!
\psi   \mapsto \{ u \otimes_A u' \mapsto \psi (u u') \},
 \\
\pl
\!\!\! &
\colon {U_*} \to A,
&\!\!
\psi  \mapsto  \psi (1_\uhhu).
\end{array}
\end{equation}
In case $\due U \lact {}$ is finitely generated $A$-projective and therefore 
${U_*}_\bract$ as well (see below), then
\begin{equation}
\label{zuschlag}
{U_*}_\bract \otimes_A  \due {{U_*}} \blact {} \to
\Hom_A(U_\bract \otimes_A \due U \lact {}, A), \quad
\psi  \otimes_A  \psi' \mapsto \{ u  \otimes_A v \mapsto
\langle \psi' , u \bract \langle \psi , u' \rangle \rangle \},
\end{equation}
is an isomorphism and the map
$\gD_r \colon \psi \mapsto \psi^{(1)} \otimes_A \psi^{(2)}$ above defines a
right (as opposed to left) coproduct, {\em i.e.}, 
\begin{equation}
\label{trattovideo}
\langle \psi ^{(2)}, u \bract \langle \psi ^{(1)} , u' \rangle \rangle = \langle \psi , u u' \rangle,
\end{equation}
and the map $\pl$ above is a (right)  counit such that the coproduct $\gD_r$ becomes counital.
Choosing a dual basis $\{e_j\}_{1 \leq j \leq n}  \in
U, \ \{e^j\}_{1 \leq j \leq n}  \in {U_*}$, one writes
\begin{equation}
\label{schizzaestrappa1}
u = \textstyle\sum_j \langle e^j, u \rangle \lact e_j ,
\end{equation}
and from there,  using \eqref{duedelue}, 
$$
\langle \psi , u \rangle
= \textstyle\sum_j \langle \psi ,  \langle e^j, u \rangle \lact e_j  \rangle
= \textstyle\sum_j \langle e^j, u \rangle \langle \psi , e_j \rangle
= \textstyle\sum_j \langle  e^j \bract \langle \psi ,  e_j \rangle, u \rangle,
$$
and therefore
\begin{equation}
\label{schizzaestrappa2}
\psi  = \textstyle\sum_j e^j \bract \langle \psi ,  e_j \rangle .
\end{equation}
Hence, as claimed above, if $\due U \lact {}$ is finitely generated $A$-projective, then ${U_*}_\bract$ is so as well.

\subsection{Duals as right/left Hopf algebroids}
\label{naemlichhier}
If the left bialgebroid $(U,A)$ is a left Hopf algebroid, then the right bialgebroid given by the right dual $(U^*,A)$ is a right Hopf algebroid, and if the left bialgebroid $(U,A)$ is a right Hopf algebroid, then the right bialgebroid given by the left dual $(U_*, A)$ is a left Hopf algebroid, as explicitly proven in \cite[\S3]{Kow:WEIABVA}:
if $\{e_i\}_{1 \leq i \leq n} \in U, \ \{ e^i\}_{1 \leq i \leq n} \in U^*$ is a dual basis for the right dual of $U$, then the translation map reads
\begin{equation*}
\phi^- \otimes_\Aopp \phi^+ := \textstyle\sum_i e^i \otimes_\Aopp (e_i \rightslice \phi)
\end{equation*}
for $ \phi \in U^*$,
where
\begin{equation*}
(u \rightslice \phi)(u')
:= \gve(\phi(u_- u') \blact u_+).
\end{equation*}
If, on the other hand, $\{e_j\}_{1 \leq j \leq n} \in U, \ \{ e^j\}_{1 \leq j \leq n} \in U_*$ is a dual basis for the left dual of $U$, then on $U_*$ one obtains the translation map as
\begin{equation}
    \label{staendigallergie1}
    \psi^\smam \otimes_A \psi^\smap := \textstyle\sum_j e^j \otimes_A (e_j \manda \psi)
\end{equation}
for $\psi \in U_*$, where
\begin{equation}
  \label{staendigallergie2}
(u \manda \psi)(u')
:= \gve(u_\smap \bract \phi(u_\smam u')).
\end{equation}
As in \cite[Lem.~3.3]{Kow:WEIABVA}, it is a direct check that
\begin{equation}
  \label{staendigallergie3}
  \langle \psi^\smam, u \rangle \lact \psi^\smap = u \manda \psi.
\end{equation}
Finally,
recall the map $S^* \colon U^* \to U_*$ from \cite[Eq.~(5.1)]{CheGavKow:DFOLHA}; one has, using \eqref{mampf3},
$$
(u \manda S^*\phi)(u') = (u' \rightslice \phi)(u)
$$
for $\phi \in U^*$, which appears somehow curious.

\end{document}